\documentclass[12pt,a4paper]{amsart}

\usepackage[utf8]{inputenc}
\usepackage[T1]{fontenc}
\usepackage{lmodern}
\usepackage{mathtools}
\usepackage{standard}
\usepackage{amsrefs}
\usepackage{enumerate}
\usepackage[usenames,dvipsnames]{xcolor}
\usepackage[colorlinks=true,linkcolor=Red,citecolor=Green]{hyperref}
\usepackage[depth=2]{bookmark}
\usepackage{tikz}
\usepackage{microtype}

\numberwithin{theorem}{section}

\newcommand{\mf}{\operatorname{mf}}
\DeclareMathOperator{\ff}{ff}
\newcommand{\fib}{\operatorname{fib}}
\newcommand{\base}{\operatorname{base}}

\DeclareMathOperator{\cl}{cl}

\newcommand{\cB}{\mathcal{B}}

\newcommand{\cK}{\mathcal{K}}

\newcommand{\cX}{\mathcal{X}}
\newcommand{\cY}{\mathcal{Y}}
\newcommand{\ellvar}{\ensuremath{\boldsymbol{\ell}}}
\newcommand{\mink}{\mathbb{M}}

\DeclareMathOperator{\coker}{coker}
\DeclareMathOperator{\diag}{diag}

\newcommand{\scl}{\mathrm{scl}}

\newcommand\Psisclsc{\Psi_{\scl, \sc}}

\newcommand\NP{\mathrm{NP}}
\newcommand\SP{\mathrm{SP}}

\newcommand{\Rad}{\mathcal{R}}
\DeclareMathOperator{\fey}{Fey}

\DeclareMathOperator{\sources}{src}
\DeclareMathOperator{\sinks}{snk}
\newcommand\fut{\mathrm{f}}
\newcommand\pas{\mathrm{p}}
\newcommand{\Kf}{K^{\fut}}

\newcommand{\Kp}{K^{\pas}}

\newcommand{\Kfp}{K^{\fut/\pas}}

\newcommand{\fibeq}{\operatorname{fibeq}}

\renewcommand{\sc}{\mathrm{sc}}
\newcommand\poles{C}
\newcommand{\tsc}{3\sc}
\newcommand\Psitsc{{}^{\tsc}\Psi}
\newcommand\Phitsc{{}^{\tsc}\Phi}
\newcommand\Psisc{{}^{\sc}\Psi}
\newcommand\Sobsc{H_{\sc}}
\newcommand\Ssc{{}^{\sc}S}
\newcommand\Tsc{{}^{\sc}T}

\newcommand\Ttsc{{}^{\tsc}T}
\newcommand\Stsc{{}^{\tsc}S}
\newcommand\WFtsc{{}^{\tsc}\!\WF}

\newcommand\Elltsc{{}^{\tsc}\!\Ell}
\newcommand\Chartsc{{}^{\tsc}\!\Char}
\newcommand\Radtsc{{}^{\tsc}\Rad}
\newcommand\esssupptsc{{}^{\tsc}\!\esssupp}
\newcommand\Wperpo{\overline{W^\perp}}

\newcommand\Tsco{{}^{\sc}\overline{T}}
\newcommand\Ttsco{{}^{\tsc}\overline{T}}
\newcommand\Diffsc{\Diff_{\sc}}
\newcommand\Difftsc{\Diff_{\tsc}}

\newcommand{\Tdot}{{}^{\sc}\dot{T}}
\newcommand{\Tdoto}{{}^{\sc}\dot{\overline{T}}}

\newcommand{\Ctscd}{\dot{C}_{\tsc}[X; C]}

\newcommand{\Hamsc}{{}^{\sc}\!H}

\newcommand\prinsymb[1]{j_{\tsc,#1}}        
\newcommand\fibsymb[1]{\sigma_{\tsc,#1}}    
\newcommand\mfsymb[1]{\hat{N}_{\mf,#1}}     
\newcommand\ffsymb[1]{\hat{N}_{\ff,#1}}     
\newcommand\ffsymbpm[1]{\hat{N}_{\ff,\pm,#1}}     
\newcommand\prinsymbz{j_{\tsc}}
\newcommand\fibsymbz{\sigma_{\tsc}}
\newcommand\mfsymbz{\hat{N}_{\mf}}
\newcommand\ffsymbz{\hat{N}_{\ff}}
\newcommand\ffsymbzpm{\hat{N}_{\ff,\pm}}

\newcommand\scprinsymb[1]{j_{\sc,#1}}        
\newcommand\scfibsymb[1]{\sigma_{\sc,#1}}    
\newcommand\scnormsymb[1]{\hat{N}_{\sc,#1}}     

\newcommand\lra{\longrightarrow}

\newcommand\normres[1]{\norm{#1}_{-N, -M}}
\newcommand\Sobres{\Sobsc^{-N, -M}}

\newcommand\Pzeroinv{(P_{0})^{-1}_{\fey}}
\newcommand\PVinv{(P_{V})^{-1}_{\fey}}

\newcommand\Qsrcpsi[1]{\tilde{Q}_{\sources, #1}}

\DeclareMathOperator{\sgn}{sgn}

\title[Feynman propagator for Klein--Gordon]{
The Klein--Gordon equation on asymptotically Minkowski spacetimes: the Feynman propagator}

\author[D. Baskin]{Dean Baskin}
\address{Dean Baskin: Department of Mathematics, Texas A\&M University \\ College Station, TX 77843, USA}
\email{dbaskin@math.tamu.edu}
\author[M. Doll]{Moritz Doll}
\address{Moritz Doll: School of Mathematical and Physical Sciences, Macquarie University \\ NSW 2109 \\ Australia}
\address{School of Mathematics and Statistics, University of Melbourne \\ VIC 3010 \\ Australia}
\email{moritz.doll@unimelb.edu.au}
\author[J. Gell-Redman]{Jesse Gell-Redman}
\address{Jesse Gell-Redman: School of Mathematics and Statistics, University of Melbourne \\ VIC 3010 \\ Australia}
\email{jgell@unimelb.edu.au}

\begin{document}
\begin{abstract}
We develop a theory of Feynman propagators for the massive Klein--Gordon equation with
asymptotically static perturbations. Building on our previous work on the causal propagators, we
employ a framework based on propagation of singularities estimates in Vasy's 3sc-calculus. We
combine these estimates to prove global spacetime mapping properties for the Feynman propagator,
and to show that it satisfies a microlocal Hadamard condition.

We show that the Feynman propagator can be realized as the inverse of a mapping between appropriate
$L^2$-based Sobolev spaces with additional regularity near the asymptotic sources of the
Hamiltonian flow, realized as a family of radial points on a compactified spacetime.
\end{abstract}

\maketitle

\tableofcontents

\section{Introduction}

Let $(M,g)$ be a globally hyperbolic, asymptotically Minkowski
spacetime as considered previously by the authors, \cites{BDGR, BVW, GRHV}.  In this work, we construct the Feynman propagator for the
Klein--Gordon operator
\begin{align*}
    P_V \coloneqq \square_g - m^2 - V\, .
\end{align*}
where $m > 0$ and $V$ is a smooth potential function with spatial
decay.  

On flat Minkowski spacetime $(\mathbb{R}^{n + 1}, g_{\mink})$, the
Feynman propagator for the free Klein--Gordon operator,
$\square_{g_{\mink}} - m^2 = D_{t}^{2} - \Delta - m^{2}$, is the
Fourier multiplier by the distribution 
$\left( \tau^2 - |\zeta|^2 - m^2 + i 0 \right)^{-1}$.  In this work,
we adopt the perspective (taken in the work of the
Gell-Redman--Haber--Vasy \cites{GRHV} and Gerard--Wrochna \cite{GW2019}) that an appropriate
generalization of the Feynman propagator for $P_V$ is an operator
which is both a global spacetime inverse for $P_V$ (e.g.~on compactly
supported distributions) and which satisfies the well-known wavefront
set property, sometimes called the ``microlocal Hadamard condition'',
detailed below.

Our main results are Theorem~\ref{thm:invertible_no_bound} and
Theorem~\ref{thm:invertible strong decay},
which prove existence of the Feynman propagator, and Theorem~\ref{thm:regularity}, which proves a variant of the microlocal Hadamard condition.

We now provide a simplified version these results.
On Minkowski space $\RR^{n+1}$, let $V = V(t,z)$ be given by
\begin{align}\label{eq:simple_potential}
    V(t, z) = V_0(z) + V_1(t, z)\,,
\end{align}
where $V_0 \in S^{-1}(\RR^n; \RR)$ exhibits symbolic decay of order $-1$ and
$V_1 \in \CI(\RR^{n+1}; \RR)$ such that
\begin{align*}
    \abs{ \pa_t^k \pa_z^\alpha V_1(t,z)} \lesssim \ang{t, z}^{-1-k} \ang{z}^{-\abs{\alpha}}
\end{align*}
for all $k \in \NN_0$ and $\alpha \in \NN_0^n$.

\begin{theorem}\label{thm:main}
    If $H_{V_0} = \Delta + m^2 + V_0$ has purely absolutely continuous spectrum, then
    there exists a bounded linear operator $\PVinv : \CcmI(\RR^{n+1}) \to \CmI(\RR^{n+1})$ such that for all $f \in \CcmI(\RR^{n+1})$,
    \begin{align*}
        P_V \PVinv f = \PVinv P_V f = f\,,
    \end{align*}
    and
    \begin{align*}
        \WF_{\cl}(\PVinv f) &\subset \WF_{\cl}(f) \cup \bigcup_{s \geq 0} \Phi_s( \WF_{\cl}(f) \cap \Char(P_0) )\,.
    \end{align*}
    where $\WF_{\cl}$ denotes the classical wavefront set of a
    distribution, $\Phi_s$ is the flow generated by the Hamilton
    vector field of the principal symbol 
    of $P_{0} = \square_{g} - m^{2}$ on the cotangent bundle.
\end{theorem}

Theorem~\ref{thm:invertible_no_bound} is in fact more general and allows for the possibility of
\textbf{first order perturbations},
as well as for \textbf{different limiting potentials} as $t\to \pm \infty$.
In the case that the limiting Hamiltonians have \textbf{bound states}, we can prove the same conclusions under slightly stronger assumptions on decay of $V - V_0$,
see Theorem~\ref{thm:invertible strong decay} for a precise statement.
The appearance of the
operator $P_{0}$ in the theorem is explained in our previous work
\cite{BDGR}.  We mention here that the characteristic sets and the
flows of $P_{V}$ and $P_{0}$ are identical away from the singular
locus of $V$ on a compactified spacetime, a set which consists of
two points, the ``poles'', described in detail below.

There are four standard propagators, or inverses, that arise in the
study of wave propagation: the forward and backward causal
propagators, which we denote by $G_{\pm}$, and the Feynman and
anti-Feynman propagators.
Our work applies equally well
to anti-Feynman propagators, which can be constructed exactly as below
but with the opposite direction of regularity propagation along the flow.

The causal propagators $G_{\pm}$ can be characterized easily by the support condition
\begin{align*}
    \supp f \subset \set{\pm t \geq T} \implies \supp G_{\pm} f \subset \set{\pm t \geq T}\,.
\end{align*}
The Feynman propagator $\PVinv$, however, does not have such a simple characterization.
The physical intuition is that $\PVinv$ propagates positive energy forward in time and negative energy backward in time.
This can be made precise either by prescribing asymptotics or using a wavefront set condition.
We describe the wavefront set condition below in Section~\ref{sec:regularity}.

The asymptotic condition is simple to describe in the case of the free
Klein--Gordon equation on Minkowski space.
There we can solve $P_0 u = f$ for $f \in \CcI(\RR^{n+1})$ using the Fourier transform and observe that solutions have
asymptotics
\begin{align}\label{eq:free asymptotics}
    u = t^{-n/2} \left( a^{\fut}_+(z/t) e^{im\sqrt{t^2 - \abs{z}^2}} + a^{\fut}_-(z/t) e^{-im\sqrt{t^2 - \abs{z}^2}}\right) \left(1 + O(1/t)\right)
\end{align}
as $t \to + \infty$ and $\abs{z}/t \leq c < 1$ for some functions
$a^{\fut}_+$ and $a^{\fut}_-$.  (Here the superscript $\fut$ stands
for ``future'', not for the forcing.)
Similarly we find $a^{\pas}_{\pm}$ for $t \to - \infty$.
\textit{The Feynman propagator produces the solution operator
  with $a^{\fut}_- \equiv 0$ and $a^{\pas}_+ \equiv 0$.
}

The relationship between the Hamiltonian flow and the asymptotic form
for Feynman solutions can be understood through consideration of the
phase functions
\[
  \phi_{\pm}(t,z) = \pm m \sqrt{t^2 - \abs{z}^2}.
\]
The graphs of these functions in the cotangent bundle, i.e., the
submanifold
\begin{align*}
    \set{ (t, z, \pa_t \phi_\pm, \pa_z \phi_\pm) : (t, z) \in \RR^{n + 1}}
\end{align*}
defines the surface
\begin{equation} 
  \set{ \tau^2 - (\abs{\zeta}^2 + m^2) = 0, \  \tau z + t \zeta = 0, \pm\tau >
    0}. \label{eq:radial-set-1}
\end{equation}

We consider a compactification $\overline{\RR^{n+1}}$ of $\RR^{n+1}$;
the boundary of this compactification is (locally) given by $x = \pm
1/t = 0$ and is parametrized by the variable $y= \pm z / t$.
The set~\eqref{eq:radial-set-1} can be restricted to the boundary;
this restriction has four components, corresponding to the choices of $\pm t
> 0$ and $\pm \tau > 0$, and can be identified precisely as the limit points of the
trajectories of the Hamiltonian flow of the Hamiltonian function
$\tau^2 - (\abs{\zeta}^2 + m^2)$, i.e., the symbol of $P_0$.

The coefficients $a^{\fut}_-, a^{\pas}_+$ correspond to the sources of the
Hamiltonian flow, $\Rad_{\sources}$.  (In terms of the sign choices
above, $\Rad_{\sources}$ corresponds to $t<0, \tau>0$ and
$t>0,\tau < 0$.)
The coefficients
$a^{\fut}_+, a^{\pas}_-$ microlocally correspond to the sinks $\Rad_{\sinks}$
($t>0,\tau>0$ and $t<0,\tau<0$); nontrivial asymptotics at the former are
excluded by the mapping properties of the Feynman propagator, while at
the latter they are allowed.

To construct the Feynman propagator, we follow the method
developed by Vasy~\cite{Vasy2013}, in which a propagator is realized as the inverse of
a Fredholm operator.  That is, we define families of
Hilbert spaces $\cX, \cY$ of tempered distributions on
$\mathbb{R}^{n + 1}$ such that $P_V \colon \cX \lra \cY$ is a Fredholm
operator, and then subsequently show that this operator is: (1)
invertible under certain assumptions on the potential $V$, and (2) its
inverse satisfies the defining properties of the Feynman propagator.

In the absence of a potential $V$, or more generally in the setting in
which $V = V(t, z)$ is a decaying scattering potential on
$\mathbb{R}^{n + 1}$ (i.e., $V$ enjoys symbolic decay jointly in space
and time), Gerard--Wrochna~\cite{GW2019} constructed the Feynman
propagator via a related method.  In the same setting,
Vasy~\cite{Vasy2018}  proved estimates
which lead directly to an alternative construction of a Feynman
propagator, which we reviewed in our prior
work~\cite{BDGR}*{Sect.~2}.

In each of these constructions, the spaces $\cX, \cY$ referred to above are subspaces of scattering
Sobolev spaces
$\Sobsc^{s, \ell}(\mathbb{R}^{n + 1}) = \ang{t, x}^{- \ell} \Sobsc^{s}(\mathbb{R}^{n + 1})$.
In Vasy's treatment~\cite{Vasy2018}
(and in our treatment of causal propagators~\cite{BDGR}) one takes $\ell = \ellvar$ to be a variable
order, i.e., a function $\ellvar = \ellvar(t, z, \tau,
\zeta)$.  Which propagator is selected depends on the properties of
$\ellvar$, namely whether it is above or below the threshold value of $-1/2$ at the
four components of the radial set.  For the Feynman propagator, one
uses $\ellvar > -1/2$ at $\Rad_{\sources}$ and $\ellvar <
-1/2$ at $\Rad_{\sinks}$.

To analyze $P_V$ for asymptotically static potentials $V$, we use an
adaptation of Vasy's $\tsc$-pseudodifferential calculus introduced in
\cites{Vasy2000,Vasy2001}.  The modifications needed to treat the
Klein--Gordon operator were presented in our prior paper \cite{BDGR}, where we proved
propagation estimates at all parts of phase space for the operators
$P_V$.  In that paper, we gave a construction of the causal
propagators, in which the range of, say, the forward causal propagator
$G_+$ is contained in a variable order weighted space
$\Sobsc^{s, \ellvar_+}(\mathbb{R}^{n + 1})$, in which
$\ellvar_+ = \ellvar_+(t, z)$ is a function on space time which
satisfies $\ellvar_+ > -1/2$ at $\iota^-$ and $\ellvar_+ < -1/2$ at
$\iota^+$, where $\iota^+, \iota^-$ are future/past timelike infinity,
respectively.  The significance of these conditions is that solutions
$G_+ f$ are allowed below threshold asymptotics only to the future.

With this method, the main distinction between the treatment of the free Klein--Gordon operator $P_0$ and the perturbed
operators $P_V$ is that the potential $V$ is not a smooth function on
the standard radial compactification
$X = \overline{\mathbb{R}^{n + 1}}$.  It fails to be smooth precisely
at the ``north/south poles'', $\NP/\SP$, depicted in
Figure~\ref{fig:fiber_np} below. To
analyze the behavior of $P_V$ near the poles, one uses the
(operator-valued) indicial
operators $\ffsymbzpm(P_V)$, which are
essentially the Fourier transform (in time) of the limiting operator.

The global nature of the indicial operator
presents a technical obstacle in that variable decay orders $\ellvar$
are incompatible
with the indicial family.  Since both $\Rad_{\sinks}$ and $\Rad_{\sources}$
have nontrivial projections at both poles, any construction of the
Feynman propagator with Vasy's method must confront this issue.  We do
so by \textit{avoiding the use of variable order spaces altogether,}
which is a key technical novelty of our construction.
Instead, we define spaces which enforce above threshold behavior at
$\Rad_{\sources}$ through the use of a microlocal cutoff.  That is, we
use microlocal cutoffs $Q_{\sources}$ supported at $\Rad_{\sources}$
to define the $\cX, \cY$ spaces and then subsequently show that the
Feynman propagator thus constructed does not depend on our choices.

As in our analysis of the causal propagators, the possible existence
of bound states for the limiting spatial Hamiltonians $H_{V_{\pm}}$
plays an important role in determining the potential existence of a
kernel or cokernel in the Fredholm map constructed above.  When
$H_{V_{\pm}}$ are positive, we are able to construct the Feynman
propagator as an operator on $\CcI(\RR^{n + 1})$ (or more generally on
distributions $\schwartz'$ satisfying an appropriate wavefront set
condition.)  This is done in Section \ref{sec:bound states}.

\subsection{Related work}

Our work generalizes the results from G\'erard--Wrochna~\cites{GW2019,
  GW2020} and the implied construction in \cite{Vasy2018}.  These
works are the closest in outline to the present work and yield Feynman propagators for Klein--Gordon operators whose
potential perturbations are ``scattering'' in spacetime; in particular
their potentials decay in spacetime.

The results of
Derezinski--Siemssen~\cites{DS2018,DS2019} are perhaps the most
similar to the present one in that they construct the Feynman
propagator in a non-trivial time-dependent setting using Kato's ``method of
evolution equations'' to construct propagators.
Related work on the existence of Feynman propagators via essential
self-adjointness of the Klein--Gordon operator by
Nakamura--Taira~\cites{NT2023,NT2023-2} and Vasy~\cite{Vasy2020}.  Moreover, the Feynman propagator
appears in index theory on Lorentzian manifolds~\cite{BS2019}.

As our results use many-body microlocal methods, the only assumptions
made in the finite time region are dynamical, namely the
assumption that the flow is non-trapping.
Therefore, the global spectral theoretic assumptions in
Derezinksi--Siemssen (see point 3 in the introduction~\cite{DS2019}) are not needed.  Interestingly, assumptions on
the point spectrum do arise in the $t \to \pm \infty$ limit.
Specifically, we obtain
a Feynman-type Fredholm problem so long as there is no point spectrum for the
limiting Hamiltonians at $0$.  We then obtain a Feynman
\textit{propagator}, i.e., inverse to the Fredholm mapping, if the
limiting Hamiltonians are positive.

Although we do not study many-body Hamiltonians directly here, since our work
uses the $\tsc$-calculus of Vasy, the analysis herein is related to
previous work on such operators.  We do not include an exhaustive
list of related work in that field here, but we note that the use of
microlocal cutoffs in many -body scattering has a long history.  We
draw particular attention to the pioneering work of 
Gérard--Isozaki--Skibsted~\cite{GIS}, which uses functions of the Hamiltonian to construct
commutators.  This strategy was later adopted by Vasy, and
informs our propagator construction, both here and in our previous work.

Constructing Fredholm problems for operators on non-compact manifolds
requires a precise analysis at infinity.  However, microlocal analysis
and in particular the calculus of Fourier integral operators has been
used by Duistermaat--Hörmander~\cite{DH1972} to prove the existence of
distinguished parametrices, which are inverses modulo smoothing
operators, for hyperbolic operators. In particular, they constructed a
Feynman parametrix. Radzikowski~\cite{Radzikowski1996} showed that the
wavefront set condition that distinguishes the Feynman parametrix is
equivalent to the Hadamard condition of quantum field theory.  We also
refer to Islam--Strohmaier~\cite{IS2012} for a more modern treatment
of distinguished parametrices which also allows for vector-valued
operators.

For other recent work using tools from the many-body analysis of Vasy
to study hyperbolic PDE, see; Hintz and Hintz--Vasy
\cites{HiTrans,HiLWAFS,HiGluing,HiVa2015,HiVa2023}.  See also the work
of Ma on many-body Hamiltonians \cite{Ma} as well as recent work of
Sussman on Klein--Gordon~\cite{Suss} using microlocal methods.

The present paper differs from the works above in at least three
significant ways.  To our knowledge, this paper is the first to construct the Feynman
propagator for asymptotically static potential without explicitly appealing to a global time-splitting or
spectral hypotheses on each time-slice.  Our construction also has the
advantage that a microlocal Hadamard condition is easily proved
(indeed, it is nearly automatic).  Finally, in contrast with other
related constructions, our construction does not use variable order Sobolev spaces.
Vasy~\cites{Vasy2021, Vasy2021-2} employs an approach
using second microlocalization to analyze $\Delta - \lambda^{2}$ in
the $\lambda \to 0^{+}$ limit.  The estimates proven in that work also
avoid the use of variable order spaces.  That setting is structurally
similar to that of the Klein--Gordon operator plus a potential,
provided the potential is ``spacetime scattering'', in particular
meaning the potential functions decay in time.

\subsection*{Outline of the paper}
The paper is organized as follows.
In Section~\ref{sec:scattering section} we consider the case that $V \equiv 0$, which puts us in the setting of scattering operators
and while it is possible in this case to use the same approach as for the causal propagator, we construct the Feynman propagator using
localizers to the radial set giving a simplified version of the proofs of the main theorems.
In Section~\ref{sec:tsc calculus} we recall the main properties of the $\tsc$-calculus introduced by Vasy~\cites{Vasy2000,Vasy2001}
and the extensions from \cite{BDGR} needed to treat the Klein--Gordon operator.
In Section~\ref{sec:tsc localizers} we define the class of localization operators that are used to define the Sobolev spaces
for the Feynman propagator and we state a slightly modified radial set estimate.
We construct the Feynman propagator in Section~\ref{sec:tsc prop construction} in the case that there are no bound states
and prove the Hadamard property in Section~\ref{sec:regularity}.
Finally, in Section~\ref{sec:bound states}, we discuss the modifications of the argument that are needed to treat the case that
the limiting Hamiltonians have bound states.

\subsection*{Acknowledgments}
This research was supported in part by the Australian Research Council
grant DP210103242 (JGR, MD) and National Science Foundation grant
DMS-1654056 (DB).  We acknowledge the support of MATRIX through the
programs ``Hyperbolic PDEs and Nonlinear Evolution Problems'' September
18-29, 2023 and
``Harmonic and Microlocal Analysis in PDEs'' December 16-20, 2024.
DB and MD were supported by the ESI through the
program ``Spectral Theory and Mathematical Relativity'' June 5-July
28, 2023.
The authors are grateful to Andrew Hassell, Kouichi Taira, Andras Vasy, and Jared Wunsch for helpful discussions.

\section{The Feynman propagator for the free Klein--Gordon equation}\label{sec:scattering section}

In this section, we consider the case $V \equiv 0$. This is simpler
than the case of general asymptotically static $V$, in that it can be
treated using only the scattering calculus of Melrose
\cite{Melrose94}.  In contrast with related treatments~\cite{Vasy2018} of this case,
including our own work~\cite{BDGR} on causal propagators, here we do not use variable order weight
functions for reasons outlined in the introduction.  Treating the case $V \equiv 0$ separately allows us to
focus on the new features of the construction, which uses
microlocalizers to define the Feynman spaces as opposed to variable
order weights.

We write
\begin{align*}
  P_0 \coloneqq \square_g - m^2,
\end{align*}
in particular we think of the background asymptotically Minkowski
metric $g$ as fixed; the subscript $0$ refers only to the potential
$V$ being identically zero.

Our construction of the Feynman propagator relies on a detailed
understanding of the Hamiltonian flow of $P_0$ on the characteristic
set, which we describe in Section \ref{sec:hamiltonian flow}.  We in
particular identify the family $\Rad_{\sources}$ of sources of the
Hamiltonian flow.  In Sections \ref{sec:microlocal cutoffs scattering}
we construct microlocal cutoffs $Q_{\sources}$ to $\Rad_{\sources}$,
and in Section \ref{sec:Feynman problems scattering} use to impose an
above threshold decay condition on distributions in the Feynman
domain, thereby making a Fredholm problem for $P_V$.  Then in Section
\ref{sec:feynman propagator scat} we discuss the Feynman propagator as
it arises as an inverse to these Fredholm problems.

\subsection{Geometry of the Hamiltonian flow}\label{sec:hamiltonian flow}

To analyze $P_0$, we use the scattering formalism of Melrose
\cite{Melrose94}.  In particular, we work on the radial
compactification
\begin{equation}
X = \overline{\mathbb{R}^{n + 1}}, \label{eq:radial compactification}
\end{equation}
a compact manifold whose boundary is defined by the vanishing of a
function $\rho_{\base} \in \CI(X)$, i.e., $\rho_{\base}^{-1}(0) = \partial X$
and $d \rho_{\base}|_{\pa X} \not = 0$.
We take
\begin{equation}
  \label{eq:spatial bdf}
  \rho_{\base} = \ang{t, z}^{-1} = \left( 1 + t^{2} + \abs{z}^{2}\right)^{-1/2} 
  \ge 1.
\end{equation}

With coordinates $t, z$ on $\mathbb{R}^{n + 1}$,
writing
\begin{equation}
  \label{eq:mink met}
  g_{\mink} = dt^{2} - dz^{2}
\end{equation}
we assume that our background metric $g$ is a non-trapping Lorentzian
metric and that the difference of the components of $g$ and $g_{\mink}$
satisfies, for all $j, k \in \{ 1, \dots, n + 1\}$,
\begin{equation}
  \label{eq:metric assumption}
  g_{jk} - (g_{\mink})_{jk} \in S^{-2}(\mathbb{R}^{n + 1}) = \rho_{\base}^{2} C^{\infty}(X).
\end{equation}
The d'Alembertian is $\square_g = - (1/\sqrt{|g|}) \partial_\mu g^{\mu
  \nu} \sqrt{|g|} \partial_\nu$, so for the Minkowski metric it is
just the free wave operator:
\begin{align*}
    \square_{g_{\mink}} = -\pa_t^2 + \sum_{j=1}^n \pa_{z_j}^2 = D_t^2
  - \Delta.
\end{align*}

The radial compactification used here is different from the Penrose
compactification of Minkowski space; in the radial compactification,
future and past causal infinity are the limit loci of forward and
backward timelike rays, respectively:
\begin{align*}
    \iota^+ &\coloneqq \set{(t,z) \colon t \geq \abs{z}} \cap \pa X\,,\\
    \iota^- &\coloneqq \set{(t,z) \colon -t \geq \abs{z}} \cap \pa
              X\,.
\end{align*}
Near $\iota^+$, we typically use coordinates $x = 1/t, y = z/t$, which
are valid in any region in which $0 \le x, |y| \le C$ where $C > 0$.
In these coordinates,
\begin{align*}
  \iota^+ = \{ x = 0, |y| \le 1 \}.
\end{align*}

The full symbol of $P_0$ is
\begin{align*}
    p(t,z,\tau, \zeta) = g^{-1}_{(t,z)}\left((\tau, \zeta),(\tau,\zeta)\right) + m^2\,.
\end{align*}
where
$g^{-1}_{(t,z)}$ is the inverse of the metric $g_{(t,z)}$.
This symbol $p$ is an example of a scattering symbol of order $(2,0)$,
as we describe now.

The scattering cotangent bundle
$\Tsc^*X = X \times \mathbb{R}^{n + 1}$ is simply the compactification
of the spacetime factor of the standard phase space
$T^* \mathbb{R}^{n + 1}$, the latter written with
coordinates $(t, z, \tau, \zeta)$.  Its fiber compactification
\begin{equation}
  \label{eq:1}
  \Tsco^* X = \overline{\mathbb{R}^{n + 1}_{t,z}} \times \overline{\mathbb{R}^{n + 1}_{\tau,\zeta}},
\end{equation}
on which the momentum/energy variables -- the ``fibers'' -- are also
radially compactified, is a manifold with corners, where the fiber
boundary is defined by the vanishing of $\rho_{\fib} \in
\CI(\overline{\mathbb{R}^{n + 1}_{\tau, \zeta}})$ where
\begin{equation}
  \label{eq:fiber bdf}
  \rho_{\fib} = \ang{\tau, \zeta}^{-1}.
\end{equation}
The classical scattering symbols can then be characterized simply by
\begin{equation}
  \label{eq:scattering symbols}
    \Ssc^{m,r}(\mathbb{R}^{n + 1}) = \rho_{\fib}^{-m} \rho_{\base}^{-r} \CI(\Tsco^* X).
  \end{equation}
The quantization of these symbols yields the space of scattering pseudodifferential
operators:
\begin{equation}
  \label{eq:scattering psidos}
 \Psisc^{m,r} =     \Op_L (\Ssc^{m,r}(X)),
\end{equation}
and the differential operators in this class are denoted
$\Diffsc^{m,r}(X)$.  Equivalently, $L \in \Diffsc^{m,r}(X)$ if an only
if
\begin{equation}
  \label{eq:differential operators}
L = \sum_{j + \abs{\alpha}\leq m}
a_{j,\alpha}(t,z) D_{t}^j D_z^{\alpha}, \quad a_{t,\alpha} \in
  S^{r}(\RR^{n + 1}).
\end{equation}

Therefore $P_0 \in \Diffsc^{2,0}$, and we can apply general scattering
calculus results to $P_0$.  In particular, we can define the
scattering principal symbol by the restriction
\begin{equation*}
  \scprinsymb{2,0}(P_0) \coloneqq \rho_{\fib}^2 \cdot p \rvert_{\partial
    \Tsco^* X},
\end{equation*}
which we write in two components corresponding to the two boundary
hypersurfaces of $\Tsco^* X$ defined by $\rho_{\fib} = 0$ (fiber
infinity) and $\rho_{\base} = 0$ (spacetime or ``base'' infinity),
\begin{equation}
  \label{eq:tsco bdry}
  \partial \Tsco^* X =    \Tsco^*_{\partial X} X \cup \Ssc^{*} X,
\end{equation}
where $\Ssc^{*}X \cong X \times \partial \overline{\RR^{n + 1}}$ denote
the sphere bundle of $\Tsc^{*} X$.
Following Melrose~\cite{Melrose94}*{Proposition 3} we write
\begin{equation}
  \label{eq:scat symbol of p0}
\scprinsymb{2,0}(P_0) = \left( \scfibsymb{2,0}(P_0), \scnormsymb{2,0}(P_0) \right).
\end{equation}
In the interior of $\mathbb{R}^{n + 1}$, $\scfibsymb{2,0}(P_0)$ is
identical to the standard principal symbol
of $P_0$, while in the interior of the fiber, i.e., for finite $(\tau,
\zeta)$, $\scnormsymb{2,0}(P_0)$ is the restriction of the total
symbol to the space time boundary $\partial X$.

Both components of $\scprinsymb{2,0}(P_0)$ are functions, and the
characteristic set is merely the vanishing locus of the symbol
restricted to the boundary, i.e., the union of the vanishing loci of
the components.  In the case of
$P_0$ this is easy to write down:
\begin{align*}
    \Char(P_0) = \set{ (q,\tau,\zeta) \colon q \in X,  g_{q}(\tau,
  \zeta) =  m^2}\,,
\end{align*}
where, when $q \in \partial X$ it is simply the condition $\tau^2 = \abs{\zeta}^2 +
m^2$, and over the interior of $X$ it is interpreted
as the equation $g_{q}(\tau, \zeta) = 0$ holding on points in the fiber boundary $\partial
\overline{\mathbb{R}^{n + 1}_{\tau, \xi}}$. 

The main object of study in these results is the Hamiltonian flow on
the characteristic set.  In the scattering setting, the Hamilton
vector field is also rescaled so that a well defined flow is induced
on the characteristic set in $\partial \Tsco^* X$.  Namely, we define the
scattering Hamilton vector field
\begin{equation}
  \label{eq:scat ham}
      \Hamsc_p \coloneqq (\rho_{\fib} / \rho_{\base}) \cdot  H_p
\end{equation}
where we use the standard definition of the Hamilton vector field
\begin{equation}
  \label{eq:ham vec gen}
      H_p \coloneqq \frac{\pa{p}}{\pa{\tau}} \pa_t +
      \frac{\pa{p}}{\pa{\zeta}} \pa_z - \frac{\pa p}{\pa t} \pa_\tau - \frac{\pa p}{\pa z} \pa_\zeta\,.
\end{equation}

We now discuss coordinates which are particularly convenient for the
computation of $\Hamsc_p$ and also for quantities of interest in later
sections.  First we note the general fact that boundary defining
functions such as $\rho_{\base}$ and $\rho_{\fib}$ are not unique.
Locally over a boundary component $\bullet \in \{\base,  \fib \}$, any
function $\tilde \rho$ satisfying
$1 / C < \tilde \rho / \rho_\bullet < C$ for $C > 0$ is also a valid boundary defining function in that
region.  For us the most convenient choices to replace $\rho_{\base}$
and $\rho_{\fib}$ are
\begin{equation}
  \label{eq:x and rho}
  x = (\sgn t)/t, \mbox{ and } \rho = (\sgn \tau)/\tau,
\end{equation}
respectively.  In regions in which $x, \rho < C$ we have coordinates
\begin{equation}
  \label{eq:simple co coordinates}
  x,\quad y = z/t, \quad \rho, \quad \mu = \zeta / \tau,
\end{equation}
and we obtain the expression
\begin{equation}
  \label{eq:scat ham2}
     (1/2)  \Hamsc_p =  -(\sgn t)(\sgn \tau) \left( x \pa_x + (\mu + y) \cdot \pa_y\right),
   \end{equation}
where, as the reader pleases, one can think of this as a redefinition
of $\Hamsc_p$ where $\rho_{\base}, \rho_{\fib}$ are replaced by $x,
\rho$, or one can think of the prefactor of
$(\rho_{\fib}/\tau)(x / \rho_{\base})$ as suppressed.  (For $a > 0$
a smooth, positive function on $\Tsco^* X$, the flow of $a \Hamsc_p$
on the characteristic set of $P_{0}$
is a smooth, non-degenerate reparametrization of the flow of
$\Hamsc_p$ and is therefore irrelevant in all statements below
regarding the Hamiltonian flow.)

The radial set is the vanishing locus of the scattering Hamilton
vector field in the boundary inside the characteristic set:
\begin{align*}
    \Rad \coloneqq \Char(P_0) \cap \Hamsc_p^{-1}(0) \subset \pa \Tsco^* X\,.
\end{align*}
The radial set of $P_0$ is computed in \cite{BDGR}*{Sect.\ 2.5}, where it is
shown that it lies over causal infinity, i.e., if $\pi \colon \Tsco^*X
\lra X$ is the projection to the base, then
\begin{align*}
  \pi(\Rad) = \iota^+ \cup \iota^-,
\end{align*}
and we write
\begin{align*}
    \Rad = \Rad^f \sqcup \Rad^p\,.
\end{align*}
for the parts of $\Rad$ lying over future/past causal infinity, meaning
\begin{align*}
    \Rad^f \coloneqq \Rad \cap \Tsco^*_{\iota^{+}} X, \quad     \Rad^p \coloneqq \Rad \cap \Tsco^*_{\iota^{-}} X\,.
\end{align*}
Both $\Rad^f$ and $\Rad^p$ consist of two connected components,
\begin{align*}
    \Rad^f = \Rad^f_+ \sqcup \Rad^f_-\,,\quad \Rad^p = \Rad^p_+ \sqcup \Rad^p_-\,,
\end{align*}
where
\begin{align*}
    \Rad^\bullet_\pm \coloneqq \Rad^\bullet \cap \set{\pm \tau \geq m}\,.
\end{align*}
For example, it is easy to show that, with $w = (\zeta / \tau) + y$,
then in the coordinates $x, w, \rho, \mu$, in the region $t > 0, \tau
> 0$,
\begin{align*}
  \Rad^f_+ = \set{x = 0,\ w = 0},
\end{align*}
where in these coordinates
\begin{align*}
(1/2)  \Hamsc_p = - x \pa_x - w \cdot \pa_w,
\end{align*}
showing that $\Rad^{f}_{+}$ is a sink of the flow.  The other three components are obtained similarly.
Finally, we denote the sources and sinks by
\begin{align*}
    \Rad_{\sources} &\coloneqq \Rad^f_- \cup \Rad^p_+\,,\\
    \Rad_{\sinks} &\coloneqq \Rad^f_+ \cup \Rad^p_-\,.
\end{align*}

The main theorem regarding the global structure of the Hamiltonian
flow on $\Char(P_0)$, which is central to the development of the
theory below is the following.
\begin{theorem}[cf.\ \cite{BDGR}*{Sect.\ 2.5}]\label{thm:global flow}
 The sets $\Rad_{\sources}$ and $\Rad_{\sinks}$ are, respectively,  global
 sources and sinks for the Hamiltonian flow on $\Char(P_0)$.
\end{theorem}

\subsection{Microlocal cutoffs}\label{sec:microlocal cutoffs scattering}

The definitions of  the spaces we use to construct the Feynman
propagator rely on microlocalization to the radial set, and we review
the basic features of this briefly.

Recall from \cite{Melrose94}*{Section 7}, that for $\Op_L(a) = A \in \Psisc^{m,r}$, the characteristic
set is in general the vanishing locus of $\scprinsymb{m,r}(A)$ in
$\partial \Tsco^* X$, and the elliptic set is its complement,
$\Ell(A) = \partial \Tsco^* X \setminus \Char(A)$.  The operator
wavefront set $\WF'(A)$ is the essential support $\esssupp(a)$ of the symbol.

For any compact subset $K \subset \partial \Tsco^* X$ and any open
neighborhood $U \subset \partial \Tsco^* X$ of $K$, a microlocal cutoff
to $K$ supported in $U$ is a $Q \in \Psisc^{0,0}$ such that
\begin{align*}
K \subset \Ell(Q) \mbox{ and } \WF'(Q) \subset U.
\end{align*}
We often assume in addition to this that, for some open neighborhood
$V$ of $K$ with $\overline{V} \subset U$, that
\begin{align*}
  \WF'(I - Q) \cap \overline V = \varnothing,
\end{align*}
which is to say that $Q$ is microlocally equal to the identity near $K$.

Microlocal elliptic regularity (cf. \cite{Vasy2018}*{Corollary 5.5}) states that if $A \in \Psisc^{m,r}$ and
$B, G \in \Psisc^{0,0}$ satisfy that $\WF'(G) \subset \Ell(A)$ and
$\WF'(B) \subset \Ell(G)$, then for any $M, N \in \RR$, there
is $C > 0$ such that, for any $s, \ell \in \RR$,
\begin{align}
  \label{eq:microlocal elliptic estimate}
  \norm{B u}_{s, \ell} \le C \left(  \norm{GAu}_{s-m, \ell-r} +   \norm{u}_{-N, -M} \right).
\end{align}

Recall the scattering wavefront set of a tempered distribution $u \in
\schwartz'(\RR^{n + 1})$, defined by its complement:
\begin{equation*}
\WF^{m,r}(u)^{c}  = \set{ q \in \partial \Tsco^{*} X \colon \exists
  A \in \Psisc^{m,r},\  q \in \Ell(A), \ Au
  \in L^{2}}.
\end{equation*}
Below we are interested in distributions $u$ which lie globally in
some Sobolev space $\Sobsc^{s, \ell_{0}}$ with $\ell_{0} < -1/2$ but lie in a better space, namely one with above threshold
decay, near $\Rad_{\sources}$.  Specifically, we work with $u$
which also satisfy, for some $\ell_{+} > -1/2$, that $\WF^{s,
  \ell_{+}}(u) \cap \Rad_{\sources} = \varnothing$.  However, this
condition is difficult to work with directly to obtain global Fredholm
estimates.  We therefore fix microlocal cutoffs to
$\Rad_{\sources}$ and work with spaces with manifest dependence on those cutoffs;
we then show that the solutions obtained by that method do not depend
on choices.

Microlocal cutoffs $Q_{\sources}$ to $\Rad_{\sources}$ are thus in
particular operators in $\Psisc^{0,0}$ such that
$\Rad_{\sources} \subset \Ell(Q_{\sources})$.  To construct them is
straightforward given the fact that $\Rad_{\sources}$ is a disjoint
union of smooth submanifolds which intersect the corner transversely.
Indeed, following the coordinate description above, the future radial
source $\Rad_{-}^f$ is the vanishing locus of the smooth coordinate
functions $x, w$ in the region where $0 \le - 1/\tau < C$ and
$|\zeta / \tau| < C$.  Thus a cutoff to this set is a quantization of
a symbol
\begin{equation}
  \label{eq:3}
  q_1 = \chi(x) \chi(|w|),
\end{equation}
where $\chi \in \CI(\mathbb{R})$ is a smooth bump function with
$\chi(s) = 1$ for $|s| < \delta / 2$ and $\chi(s) \equiv 0$ for $|s|
\ge \delta$.  If $q_2$ is the analogous symbol supported at the other
component of the sources, $\Rad^p_+$, then we can use
\begin{align}
  \label{eq:scattering source cutoff}
  Q_{\sources} = \Op_L(q_1) + \Op_L(q_2).
\end{align}
Thus we can take cutoffs $Q_{\sources}$ supported
in arbitrarily small neighborhoods of $\Rad_{\sources}$ by taking
$\delta > 0$ small in this definition.

\subsection{Sobolev spaces for the Feynman problem}\label{sec:Feynman
  problems scattering}

Fix a microlocalizer to the source,
$Q_{\sources}$.  For $s, \ell_-, \ell_+ \in \mathbb{R}$ with
$\ell_- < -1/2, \ell_+ > -1/2$, we estimate the quantity
\begin{equation*}
    \norm{u}_{s,\ell_0} +
    \norm{Q_{\sources} u}_{s,\ell_+}
\end{equation*}
in terms of appropriate norms of $P_{0} u$ and weaker norms of $u$.
Note that the finiteness of the displayed quantity implies that $u \in
H^{s, \ell_{0}}$ globally, and in addition that $\WF^{s, \ell_{+}} (u)
\cap \Rad_{\sources} = \varnothing$.  In particular, as we 
discuss below, such a distribution $u$ is above the threshold weight at the radial sources.

Digressing briefly, it is convenient to introduce notation for a
Hilbert space with a norm equivalent to the sum of these norms of $u$
just given. Indeed, for $\ell_0 < \ell_+$ and $A \in \Psisc^{0,0}$
(and, in later sections, $A \in \Psitsc^{0,0}$),
we set
\begin{align}
  \label{eq:sobolev-with-operator}
    H_A^{s,\ell_0, \ell_+}
    \coloneqq \{ u \in \Sobsc^{s,\ell_0} \colon A u \in \Sobsc^{s,\ell_+}\}\,.
\end{align}
This is complete with respect to the norm
\begin{align*}
    \norm{u}_{A,s,\ell_0, \ell_+}^2 \coloneqq
    \norm{u}_{s,\ell_0}^2 +
    \norm{A u}_{s,\ell_+}^2 \,.
\end{align*}
We observe that
\begin{align*}
    \Sobsc^{s,\ell_+} \subset H_A^{s,\ell_0,\ell_+} \subset \Sobsc^{s,\ell_0}\,.
\end{align*}

To state the global Fredholm estimate, now assume that we are given an
additional microlocalizer to the sources $Q_{\sources}' \in
\Psisc^{0,0}$ such that
\begin{equation}\label{eq:Qsrc condition}
    \Rad_{\sources} \subset \Ell(Q_{\sources}), \quad \WF'(Q_{\sources})
    \subset \Ell(Q_{\sources}'), \quad \WF'(Q_{\sources}') \cap
    \Rad_{\sinks} = \varnothing.
  \end{equation}
We must require, moreover, that $\Ell(Q_{\sources}')$ contains all
flow segments whose endpoints lie in $\WF'(Q_{\sources})$.  That is, if $q \in \WF'(Q_{\sources}) \cap
\Char(P_0)$, then we must have $\Phi_s(q) \in \Ell(Q_{\sources}')$ for all $s <
0$.  This  holds automatically, for example, if $\Rad_{\sources} \subset \Ell(Q_{\sources})$
and $\WF'(Q_{\sources})$ is convex under the flow on $\Char(P_0)$, as we will
arrange in our construction.  (This avoids the possibility of a flow
line exiting both $\WF'(Q_{\sources})$ and $\Ell(Q_{\sources}')$, in which
case propagation of singularities would not apply.)

  With these assumptions on $Q_{\sources}, Q_{\sources}'$, we have a
  global estimate, which says that, for some $C > 0$,
\begin{equation}
  \label{eq:global fredholm estimate}
  \norm{u}_{Q_{\sources},s,\ell_0, \ell_+} \le   C \left( \norm{P_0
    u}_{Q_{\sources}',s - 1,\ell_0 + 1, \ell_+ + 1} +
  \normres{u}\right) .
\end{equation}
We briefly discuss this estimate heuristically here, but
we elide its formal proof as we prove a similar estimate below in the
case $V \not \equiv 0$, see Lemma~\ref{lem:finite-dim_kernel}.

As with other estimates similar to \eqref{eq:global fredholm estimate}, this
estimate is obtained by combining three types of estimates which hold
at different parts of phase space: 1) elliptic estimates, which are
used away from the characteristic set, 2) real
principal-type propagation estimates, which are used on the
characteristic set away from the radial sets, and 3) radial points
estimates, which hold near the radial set and themselves come in two
varieties, one for above threshold decay and the other for below
\cite{BDGR}*{Sect.~2.7 and Sect.~7}.

The above threshold radial points estimates \cite{BDGR}*{Proposition~2.11} imply that, for
any $\ell_+'$ with $-1/2 < \ell_+' < \ell_+$ and any $N \in \mathbb{R}$,
\begin{equation}
  \label{eq:above thresh scattering}
  \norm{Q_{\sources} u}_{s, \ell_+} \lesssim   \norm{Q'_{\sources} P_0
    u}_{s - 1, \ell_+ + 1} + \norm{Q_{\sources}'  u}_{-N, \ell_+'} + \normres{u}
\end{equation}
This can be read as saying that $u$ is controlled in $\Sobsc^{s,
  \ell_+}$ in some neighborhood $U$ of $\Rad_{\sources}$ by $P_0 u$ in $\Sobsc^{s - 1,
  \ell_+ + 1}$ on a neighborhood $U'$ of $\Rad_{\sources}$ with
$\overline U \subset U'$, provided $u$ lies in an above threshold
space near $\Rad_{\sources}$.  Obtaining the global estimate
\eqref{eq:global fredholm estimate} from \eqref{eq:above thresh
  scattering} and the other propagation estimates is at this point a
standard argument; see, among many others, the arguments in
\cite{BDGR}*{Sect.\ 8}.  We also review this in greater detail in
Section \ref{sec:tsc prop construction}.

Following Vasy's development, natural Sobolev spaces corresponding to
the global Fredholm estimates are those in which both the left and the
right hand sides of  \eqref{eq:global fredholm estimate} are finite.
We thus define:
\begin{align*}
    \cX^{s, \ell_{0}, \ell_{+}} \coloneqq \{ u \in H_{Q_{\sources}}^{s,\ell_0,\ell_+} \colon P_V u \in H_{Q_{\sources}'}^{s-1,\ell_0 + 1, \ell_+ + 1}\}\,,\quad
    \cY^{s, \ell_{0}, \ell_{+}} \coloneqq H_{Q_{\sources}'}^{s,\ell_0,\ell_+}\,.
\end{align*}
We prefer the notation $\cX^{s, \ell_{0}, \ell_{+}}$ to the more
cumbersome $\cX^{s, \ell_{0}, \ell_{+}}_{Q_{\sources}, Q_{\sources}'}$
despite the obvious dependence of these spaces on the choices
$Q_{\sources}, Q_{\sources}'$.  We prove the following:
\begin{theorem}\label{thm:scattering Fredholm}
  For $s, \ell_{0}, \ell_{+} \in \mathbb{R}$, $\ell_{0} < -1/2,
  \ell_{+} > -1/2$, and $Q_{\sources}, Q_{\sources}' \in \Psisc^{0,0}$
  satisfying \eqref{eq:Qsrc condition}.  Then the map
  \begin{equation}
    \label{eq:Fred map scat}
    P_{0} \colon \cX^{s, \ell_{0}, \ell_{+}} \lra
    \cY^{s - 1, \ell_{0} + 1, \ell_{+}  +1}
  \end{equation}
is Fredholm.
\end{theorem}
\begin{proof}
  This proof is covered in detail in Section \ref{sec:tsc prop construction} below in
  the more general case of nonzero potential $V$.  In brief summary,
  the fact that the operator is bounded follows immediately from the
  definitions of the spaces.  Moreover, the fact that the mapping has
  closed range and finite dimensional kernel follows from the estimate \eqref{eq:global fredholm
    estimate} by a standard argument using the compactness of $H^{s,
    \ell_{0}, \ell_{+}}_{Q_{\sources}}$ in $\Sobres$ for
  sufficiently large $M,N$.  The cokernel is then identified with
  the kernel of $P_{0}^{*}$ with domain $(\cY^{s,\ell_{0},
    \ell_{+}})'$, whose elements have above threshold decay at
  $\Rad_{\sinks}$.  Similar estimates apply to those distributions,
  giving the finite dimensionality by the same compactness argument.  
\end{proof}

\subsection{The Feynman propagator and the wavefront set
  condition}\label{sec:feynman propagator scat}

The fact that the mapping~\eqref{eq:Fred map scat} is in fact
invertible (seen below in Section~\ref{sec:tsc prop construction})
provides the basis of our definition of the Feynman propagator.
However, to uniquely define the Feynman
propagator, we must ensure that the solutions $u$ to $P_{0} u = f$
obtained by applying the inverse mapping in \eqref{eq:Fred map scat}
do not depend on any choices made in the construction.  Specifically,
a given $f$ lies in many spaces $\cY^{s, \ell_{0}, \ell_{+}} = H^{s, \ell_{0},
  \ell_{+}}_{Q_{\sources}'}$ for different values of $s, \ell_{0},
\ell_{+}$ and different choices of cutoff $Q_{\sources}'$, and we must
show that the inverse mapping gives the same distribution independent
of these choices and the choice of $Q_{\sources}$ in the definition of
$\cX^{s, \ell_{0}, \ell_{+}}$.  This, and more, is summarized in the
following theorem.
\begin{theorem}\label{thm:Feynman propagator scattering}
Under the assumptions of Theorem \ref{thm:scattering Fredholm}, the mapping
\eqref{eq:Fred map scat} is invertible.

The inverse $\Pzeroinv$ is defined independently of the
parameters $s, \ell_{0}, \ell_{+}$ and the cutoffs $Q_{\sources},
Q_{\sources}'$ in the following sense.  Let $u_{1}, u_{2} \in
\schwartz'$ satisfy
\begin{equation*}
    u_{i} \in H^{s_{i}, \ell_{0,i}, \ell_{+,i}}_{Q_{\sources, i}} \mbox{
    for some } s_{i}, \ell_{0,i}, \ell_{+,i} \in \mathbb{R},
  \ell_{+,i} > -1/2,
\end{equation*}
with $\Rad_{\sources} \subset \Ell(Q_{\sources, i})$.  Then
\begin{equation*}
  P_{0} u_{1} = P_{0} u_{2} \implies u_{1} = u_{2}.
\end{equation*}

In particular, the mapping
\begin{equation}
  \label{eq:sc prop compact supp}
  \Pzeroinv \colon \CcmI(X) \lra \CmI(X)
\end{equation}
is well defined and satisfies the following Hadamard property; for $f 
\in \CcmI$,
\begin{equation}
  \label{eq:sc Hadamard}
        \WF(\Pzeroinv f) \subset \WF(f) \cup \bigcup_{s \geq 0} \Phi_s \left( \WF(f) \cap \Char(P_0) \right) \cup \Rad_{\sinks}\,.
      \end{equation}
where $\Phi_{s}$ is the Hamiltonian flow on $\partial \Tsco^{*} X$.
    \end{theorem}

\section{\texorpdfstring{$\tsc$-calculus}{3sc-calculus} for the
  Klein--Gordon operator}\label{sec:tsc calculus}

\subsection{Basics of the \texorpdfstring{$\tsc$-calculus}{3sc-calculus}}

In this section, we recall the basics of the $\tsc$-calculus first introduced by Vasy~\cite{Vasy2000} with adaptions to treat the Klein--Gordon equation in \cite{BDGR}.
On the spacetime compactification $X = \overline{\RR^{n+1}}$ discussed
in Section~\ref{sec:hamiltonian flow}, we set
$\poles = \set{\NP, \SP}$, where 
\begin{align*}
    \NP = \pa X \cap \overline{\set{z = 0}} \cap \overline{\set{t > 0}}\,,\quad 
    \SP = \pa X \cap \overline{\set{z = 0}} \cap \overline{\set{t < 0}}\,.
\end{align*}
Thus $\NP \in \iota^+$ and $\SP \in \iota^-$, and due to their
placement we refer to them as the north and south poles, respectively.

To use Vasy's three-body framework, we consider the blown-up space $[X; \poles]$ with the canonical blow-down map
\begin{align*}
    \beta_{\poles} : [X; C] \to X.
\end{align*}
The space $[X; \poles]$ is a manifold with corners with three boundary hypersurfaces,
\begin{align*}
    \ff_+ \coloneqq \beta_{\poles}^*(\NP)\,,\quad
    \ff_- \coloneqq \beta_{\poles}^*(\SP)\,,\quad
    \mf \coloneqq \beta_{\poles}^*(\pa X)\,,
\end{align*}
depicted in Figure \ref{fig:fiber_np}.  
The resolution $[X;\poles]$ has the important property that, while the potential $V$ defined in \eqref{eq:simple_potential} is not smooth on
$X$, it is smooth on $[X; \poles]$.

\begin{figure}
    \begin{minipage}{0.3\textwidth}
    \centering
    \begin{tikzpicture}[scale=2]
        \draw (1,0) arc (0:80:1cm);
        \draw (1,0) arc (0:-80:1cm);
        \draw (-1,0) arc (180:260:1cm);
        \draw (-1,0) arc (180:100:1cm);
        \draw[thick] (0.182,0.985) arc (0:-180:0.182cm); 
        \draw[thick] (0.182,-0.985) arc (0:180:0.182cm); 
        \draw (0,0.65) node {$\ff_+$};
        \draw (0,-0.65) node {$\ff_-$};
    \end{tikzpicture}
    \end{minipage}%
    \begin{minipage}{0.1\textwidth}
        \centering
        $\xrightarrow{\beta_C}$
    \end{minipage}%
    \begin{minipage}{0.3\textwidth}
    \centering
    \begin{tikzpicture}[scale=2]
        \draw (0,0) circle (1cm);
        \filldraw (0,1) circle (0.5pt) node [below] {$\NP$};
        \filldraw (0,-1) circle (0.5pt) node [above] {$\SP$};
        \draw (0.8, 0.7) node [right] {$\abs{y} = 1$};
        \draw[dotted] (0.707, 0.707) -- (-0.707, -0.707);
        \draw[dotted] (0.707, -0.707) -- (-0.707, 0.707);
    \end{tikzpicture}
    \end{minipage}%
    \caption{The blow-down map $\beta_C : [X;C] \to X$.}
    \label{fig:fiber_np}
\end{figure}

Our analysis below is based on the realization of $P_V$ as a
$\tsc$-operator.  To review, the space of differential operators in
the $\tsc$-calculus is given by
\begin{align*}
    \Difftsc^m(X) \coloneqq \Diffsc^m(X) \otimes_{\CI(X)} \CI([X; \poles])\,.
\end{align*}
Thus, a differential operator $L$ in $\Difftsc^m(X)$ is given by
\begin{align*}
    L = \sum_{\abs{\alpha} + k \leq m} a_{k,\alpha} D_t^k D_z^\alpha\,,
\end{align*}
where $a_{k,\alpha} \in \CI([X;\poles])$.
It is easy to verify that $P_V \in \Difftsc^2(X)$ if $V \in S^{-1}(\RR^n_{z})$.
More generally, we define weighted differential operators as
\begin{align*}
    \Difftsc^{m,r}(X) \coloneqq \ang{t,z}^r \Difftsc^m(X)\,.
\end{align*}

The compactified $\tsc$-cotangent bundle is defined to be the pullback bundle
\begin{align*}
    \Ttsco^* [X; \poles] \coloneqq \beta_{\poles}^* \Tsco^* X = [X; C]
  \times \overline{\mathbb{R}^{n + 1}_{\tau,
      \zeta}}\,.
\end{align*}
This is a manifold with corners with \emph{four} boundary hypersurfaces,
\begin{align*}
    \Ttsco^*_{\ff_\pm} [X; \poles]\,, \quad 
    \Ttsco^*_{\mf} [X; \poles]\,, \quad
    \Stsc^*[X; \poles]\,,
\end{align*}
with the latter being the fiber boundary
$[X; C] \times \partial \overline{\mathbb{R}^{n + 1}_{\tau, \zeta}}$,
which can be identified with the sphere bundle of
$\Ttsc^* [X; \poles]$.  The corresponding boundary defining functions are denoted
by
\begin{align*}
    \rho_{\ff_{\pm}}\,, \quad 
    \rho_{\mf}\,, \quad 
    \rho_{\fib}\,.
\end{align*}
Moreover, we define the functions
\begin{align*}
    \rho_{\ff} \coloneqq \rho_{\ff_+} \rho_{\ff_-}\,, \quad
    \rho_{\infty} \coloneqq \rho_{\ff} \rho_{\mf}\,.
\end{align*}
Here $\rho_{\infty} = \beta_{C}^{*} \rho_{\base}$ defines the boundary
of $[X ; C]$.

The space of classical $\tsc$-symbols is
\begin{align*}
    \Stsc^{m,r}(X; \poles) \coloneqq \rho_{\fib}^{-m} \rho_{\infty}^{-r} \CI(\Ttsco^* [X; \poles])\,.
\end{align*}
and the $\tsc$-pseudodifferential operators are quantizations of such symbols,
\begin{align*}
    \Psitsc^{m,r} \coloneqq \Op_L(\Stsc^{m,r})\,.
\end{align*}

Let $A \in \Psitsc^{m,r}$ and $s, \ell \in \RR$, then
(\cite{Vasy2000}*{Cor.~8.2})
\begin{align*}
    A : \Sobsc^{m+s, r+\ell}(X) \to \Sobsc^{s,\ell}(X)
\end{align*}
is a bounded mapping.  Below, we recall additional facts about
$\tsc$-pseudodifferential operators necessary for the statements and
proofs of our estimates.  Our first goal, however, is to recall the principal
symbol construction, which differs from that of the scattering
calculus as it includes the indicial operator.

If $A = \Op_L(a) \in \Psitsc^{0,0}$, then we define the symbols over $\mf$ and fiber infinity as
\begin{align*}
    \mfsymbz(A) &= a|_{\Ttsco^*_{\mf} [X;\poles]} \in \CI(\Ttsco^*_{\mf} [X;\poles])
    \intertext{and}
    \fibsymbz(A) &= a|_{\Stsc^* [X;\poles]} \in \CI(\Stsc^* [X;\poles])\,.
\end{align*}
In particular, away from $\ff_{\pm}$, the symbol of a $\tsc$ operator
is equal to its symbol in the standard scattering sense.

The indicial operator $\ffsymbz$ is a family of operators parametrized by
the vector bundle
\begin{align*}
    W^\perp \coloneqq \operatorname{span}_{\RR} \left( \frac{dx}{x^{2}} \right) \subset \Tsc_{\poles}^* X
\end{align*}
and we have the orthogonal projection
\begin{align*}
    \pi^\perp : \Tsc_{\poles}^* X \to W^\perp\,.
\end{align*}
The space $W^\perp$ is simply two copies of the line $\RR_{\tau}$
corresponding to the forms  $\tau dt$ over $\NP$ and $\SP$ and we write $W^\perp_{\pm}$ for the restriction of $W^\perp$ to $\NP$ and $\SP$, respectively.
We  also need the compactification of $W^\perp$,
\begin{align*}
    \Wperpo \cong \set{\pm \infty} \times \overline{\RR}\,.
\end{align*}
We describe below that the indicial operator is a semiclassical scattering
operator with semiclassical parameter $1/|\tau|$ as $\tau \to \pm \infty$, and its
semiclassical principal symbol is equal to its total symbol's value at
$\tau = \pm \infty$.

Given $A \in \Psitsc^{m,0}$, we now recall the indicial operator $\ffsymbz(A)$ (cf. \cite{Vasy2000}*{Chap.~6}, \cite{BDGR}*{Sect.~4.2}).
As the constructions at $\ff_{\pm}$ are identical, we work at $\ff_+$ and
write $\ff = \ff_+$.  We reintroduce the $\pm$ notation when
it is useful below. Each $A\in \Psitsc^{m,0}$ defines an operator $A_{\ff} \in \Psisc^{m,0}(\ff)$ by
\begin{align*}
    A_{\ff}(f) = (Au)_{\ff}\,,
\end{align*}
where $u \in \CI(X)$ is any extension of $f$ in the sense that $u|_{\ff} = f$.
The indicial operator on $\ff$ is then defined by 
\begin{align*}
    \ffsymbz(A)(\tau_0) \coloneqq \left( e^{i\tau_0/x} A e^{-i\tau_0/x}\right)_{\ff}\,.
\end{align*}
As an example, we recall that if $P_V = D_t^2 - (\Delta + m^2 + V)$ and $V_0 \in S^{-1}(\RR^n_z)$ with $V(t, z) - V_0(z) \in \Psitsc^{0, -1}$, then
\begin{align*}
    \ffsymbz(P_V)(\tau) = \tau^2 - H_{V_0}\,,
\end{align*}
so $\ffsymbz(P_V)$ is the partial Fourier transform of $P_{V_0}$ in
the $t$ variable.

The principal symbol of an operator $A \in \Psitsc^{m,r}$ is described
in detail in our previous paper \cite{BDGR}*{Sect.\ 4}.  It consists of
two ``local'' components, which are rescaled restrictions of the symbol of $A$
to the boundary components $\Stsc^* [X;\poles]$ and $\Ttsco^*_{\mf}
[X;\poles]$, respectively.  The former is the restriction to fiber
infinity, and its restriction to the interior  $X^{\circ} =
\mathbb{R}^{n + 1}$ is the standard principal symbol.  The rescaled
restriction to $\Ttsco^*_{\mf} [X;\poles]$ corresponds with the
spacetime infinity scattering principal symbol, i.e., it is the $\sc$,
as opposed to $\tsc$, symbol, which is well defined on $\mf^{\circ}$.
We use $\rho_{\fib} = \ang{\tau, \zeta}^{-1}$ and $\rho_{\base} = x =
1/t$:
\begin{align*}
    \mfsymb{m,r}(A) &\coloneqq \ang{\tau, \zeta}^{-m} x^r a|_{\Ttsco^*_{\mf} [X;\poles]}\,,\\
    \fibsymb{m,r}(A) &\coloneqq \ang{\tau, \zeta}^{-m} x^r a|_{\Stsc^* [X;\poles]}\,,\\
    \ffsymbpm{r}(A) &\coloneqq \ffsymbzpm(x^r A)\,.
\end{align*}
The $\tsc$-principal symbol consists of these four components,
and we denote the principal symbol mapping by
\begin{align*}
    \prinsymb{m,r} : \Psitsc^{m,r} &\to \CI(\Ttsc^*_{\mf} [X;\poles]) \times \CI(\Stsc^* [X;\poles]) \times \CI(\RR_\tau; \Psisc^{m,0}(\ff_{\pm}))\\
    A &\mapsto \left( \fibsymb{m,r}(A), \mfsymb{m,r}(A), \ffsymbpm{r}(A) \right)\,.
\end{align*}
The indicial operator is equal to the quantization of the restriction
of the symbol of $A$ to the slice $\tau$.  Specifically, if $a \in
\Stsc^{m,r}$ satisfies, at $\ff = \ff_+$,
\begin{equation*}
a_{\ff} \coloneqq (x^{r} a)   \rvert_{\Ttsco_{\ff}^{*} [X;C]},
\end{equation*}
then, \cite{BDGR}*{Lemma~4.4}
\begin{equation}
  \label{eq:2}
  \ffsymb{r}(A)(\tau) = \Op_{L,z}(a_{\ff}(z, \tau, \zeta)).
\end{equation}
Thus, the principal symbol satisfies the obvious matching conditions,
namely that the restriction of each component of the symbol the
boundary of its domain matches the restriction of the other
components of the symbol there.  Matching is the only condition
required for quantization.  This is summarized in the following proposition.

\begin{proposition}[\cite{BDGR}*{Proposition~4.6 and Proposition~4.8}]
    The kernel of the principal symbol mapping is $\Psitsc^{m-1, r-1}$ and the image is the set of those $(q_1, q_2, \set{Q_\tau})$
    such that if $q_0$ denotes the left reduction of $Q_\tau$, we have $\ang{\tau, \zeta}^{-m} q_{0} \in \CI(\ff \times \overline{\RR_{\tau,\zeta}^{n+1}})$ and the matching conditions
    \begin{gather*}
        q_1|_{\Stsc^*_{\mf}[X;\poles]} = q_2|_{\Stsc^*_{\mf}[X;\poles]}\,, \quad
        \ang{\tau, \zeta}^{-m} q_0|_{\Stsc^*_{\ff}[X;\poles]} = q_1|_{\Stsc^*_{\ff}[X;\poles]}\,,\\
        \ang{\tau, \zeta}^{-m} q_0|_{\Ttsco^*_{\ff \cap \mf}[X;\poles]} = q_1|_{\Ttsco^*_{\ff \cap \mf}[X;\poles]}
    \end{gather*}
    hold.

    Moreover, $\prinsymbz$ is multiplicative in the sense that if $A \in \Psitsc^{m_1, r_1}$ and $B \in \Psitsc^{m_2, r_2}$, then $A B \in \Psitsc^{m_1 + m_2, r_1 + r_2}$ and
    \begin{align*}
        \prinsymb{m_1 + m_2, r_1 + r_2}(AB) = \prinsymb{m_1, r_1}(A) \prinsymb{m_2, r_2}(B)\,.
    \end{align*}
\end{proposition}

\subsection{Elliptic and wavefront sets}

The appropriate notions of operator wavefront set, elliptic set, and
characteristic set, are influenced by the spatially global nature of
the indicial operator.  In particular, given $A \in \Psitsc^{m,r}$ and
$\tau \in W^{\perp}$, we have: (1) $\tau_{0} \not \in \WFtsc(A)$ if,
for some $\epsilon > 0$,
the total symbol of $A$ vanishes to all orders at the entire
lens of slices $\ff \times \{ (\tau, \zeta) : |\tau - \tau_{0}| < \epsilon \}
\subset \Ttsco^{*}_{\ff}[X;C]$, and (2) $\tau_{0} \in \Elltsc(A)$ if
$\ffsymb{r}(A)(\tau_{0})$ is an invertible operator.

Moreover, as $\tau \to \pm \infty$, $\ffsymb{r}(A)(\tau_{0})$ is
semiclassical in $h = 1/|\tau|$, and its principal symbol at
$\tau \to \pm \infty$ is the restriction of the total symbol to the ``hemisphere''
\begin{equation*}
  UH_{\pm} = \Stsc^{*}_{\ff}[X;C] \cap \cl(\{ \pm \tau > 0 \}).
\end{equation*}
In particular, as $\tau \to + \infty$,
\begin{equation}
  \label{eq:5}
  \sigma_{\scl, h = 1/\tau}(\ffsymb{r}(A)(1/h)) = \ang{\tau, \zeta}^{-m} a_{\ff} \rvert_{UH_{+}}.
\end{equation}

At a basic level, the principal symbol is a function on a
compactification of the compressed cotangent bundle
$\Tdot^* X  = (\Tsc^{*}X \setminus \Tsc^{*}_{C} X ) \cup W^{\perp}$,
specifically on 
    \begin{equation*}
    \Tdoto^* X = (\Tsco^{*}X \setminus \Tsco^{*}_{C} X ) \cup \Wperpo.
  \end{equation*}
The domain of principal symbol, $\Ctscd$, is a subset of the boundary faces of $\Tdoto^* X$:
\begin{align*}
    \Ctscd \coloneqq \Ssc_{X \setminus \poles}^* X \cup  \Tsc^*_{\pa X \setminus \poles} X \cup \Wperpo.
\end{align*}
To describe the operator wavefront set, we define a mapping
$\gamma_{\tsc}$, which associates to each point in the domain of the
principal symbol the appropriate subset of phase space which determines
the symbolic behavior:
\begin{align}\label{eq:gamma tsc easy}
\gamma_{\tsc} : \Ctscd \lra
\mathcal{P}(\pa \Ttsco^* [X;C])
\end{align}
as
\begin{align*}
    \gamma_{\tsc}(p) &= \{p\}&{} &\mbox{ for } p \in  \Ssc^{*}_{X \setminus \poles} X \cup
  \Tsco^{*}_{\pa X \setminus \poles} X\,, \\
    \gamma_{\tsc}(\tau) &= \beta_{C}^{-1} (\pi^\perp)^{-1}\{\tau\}&{} &\mbox{ for } \tau \in W^\perp\,,\\
    \gamma_{\tsc}(\pm \infty) &= UH_{\pm}&{} &\mbox{ for } \pm \infty \in \partial \Wperpo\,.
\end{align*}
In particular, $\gamma_{\tsc}$ maps $\tau_0 \in W_+^\perp$ to the
entire slice $\ff_+ \times \overline{\set{(\tau_0, \zeta) \colon \zeta
      \in \RR^n}}$.
 By abuse of notation, for a set $S \subset \Ctscd$ we write
\begin{align}\label{eq:gamma_tsc}
    \gamma_{\tsc}(S) \coloneqq \bigcup_{p \in S} \gamma_{\tsc}(p)\,.
\end{align}
We used above that $\Ssc^{*}_{X \setminus \poles} X \cup
  \Tsco^{*}_{\pa X \setminus \poles} X$ is naturally identified with $\Stsc^{*}_{[X;C] \setminus \ff} [X;C] \cup
  \Ttsco^{*}_{\pa [X;C] \setminus \ff} [X;C]$ via the blow down map.

For a symbol $a \in \Stsc^{m,r}$, we define the essential support of $a$,
$\esssupptsc(a) \subset \pa \Ttsco^* [X;C]$, by declaring
$p \in \esssupptsc(a)^c$ if and only if there exists $U \subset
\Ttsco^* [X;C]$ open and $\chi \in \CcI(\Ttsco^* [X;C])$ such that $p
\in U$, $\chi|_U \equiv 1$ and $\chi a \in \Stsc^{-\infty, -\infty}$.

\begin{definition}\label{def:WFtsc}
    Let $A = \Op_L(a) \in \Psitsc^{m,r}(X)$.
    The operator wavefront set
    \begin{align*}
        \WFtsc'(A) \subset \Ctscd
    \end{align*}
is defined as
    follows: a point $p \in \Ctscd$ is \textit{not }in the wavefront set,
    \begin{align*}
        p \in \WFtsc'(A)^c \text{ if and only if }
        \gamma_{\tsc}(p) \cap \esssupptsc(a) = \varnothing\,.
    \end{align*}

    Moreover, we define
    \begin{align*}
        \WF_{\fib}'(A) &\coloneqq \WFtsc'(A) \cap \Ssc^{*}_{X \setminus \poles} X \,,\\
        \WF_{\mf}'(A) &\coloneqq \WFtsc'(A) \cap \Tsco^*_{\pa X \setminus \poles}X \,, \\
        \WF_{\ff}'(A) &\coloneqq \WFtsc'(A) \cap \Wperpo\,.
    \end{align*}
\end{definition}
We can write the complements of each of the components as
\begin{align*}
    \WF_{\fib}'(A)^c &= \{ \alpha \in \Ssc^{*}_{X \setminus \poles} X\colon \exists U \subset \Ssc^{*}_{X \setminus \poles} X \text{ open such that $\alpha \in U$} \\
    &\phantom{= \{} \text{and $a(A)$ vanishes to infinite order on $\overline{U}$}\}\,,\\
    \WF_{\mf}'(A)^c &= \{ \alpha \in \Tsco^*_{\pa X \setminus \poles}X \colon \exists U \subset \Tsco^*_{\pa X \setminus \poles}X \text{ open such that $\alpha \in U$} \\
    &\phantom{= \{} \text{and $a(A)$ vanishes to infinite order on $\overline{U}$}\}\,,\\
        \WF_{\ff}'(A)^c &=
        \{ \tau \in W^\perp \colon \exists\ \epsilon > 0 \text{ such that $a(A)$
                             vanishes to } \\
  &\phantom{= \{} \text{ infinite order on }
  \beta_{C}^{-1}(\pi^\perp)^{-1}[\tau - \epsilon, \tau + \epsilon]\}\,  \\
  &\phantom{= \{} \cup \{ \pm \infty :  \exists  \text{ open $U
    \subset \partial \Ttsco^*[X;C]$ such
    that } UH_\pm  \subset U \\
  & \phantom{= \{} \text{ and $a(A)$ vanishes to infinite order on }
    \overline{U} \}.
\end{align*}

\begin{lemma}\label{lem:WFtsc_operator}
    Let $A \in \Psitsc^{m_1, r_1}$, $B \in \Psitsc^{m_2, r_2}$, then
    \begin{align*}
        \WFtsc'(AB) \subset \WFtsc'(A) \cap \WFtsc'(B)\,,\\
        \WFtsc'(A) = \varnothing \Rightarrow A \in \Psitsc^{-\infty, -\infty}\,.
    \end{align*}
\end{lemma}
\begin{proof}
    The first property follows from microlocality of the composition and
    the second claim easily follows from the fact that $\gamma_{\tsc}$ is surjective meaning that $\gamma_{\tsc}(\Ctscd) = \pa \Ttsco^*[X; \poles]$.
\end{proof}

Now we define the elliptic sets.  Over $\mf$ and fiber infinity, the
definition of ellipticity is exactly as in the standard scattering
case, i.e., non-vanishing (or, for operators acting on sections of
vector bundles, invertibility) of the principal symbol.  Over $\ff$ in
$W^\perp$, the correct notion of ellipticity is invertibility
between appropriate scattering Sobolev spaces.

To define the elliptic set, we note that the two components of the
symbol $\fibsymb{m, r}(A)$ and $\mfsymb{m, r}(A)$ define, by
restriction, functions on $\Ssc_{X \setminus \poles}^* X$ and
$\Tsco_{\pa X \setminus \poles}^* X$, respectively.
\begin{definition}\label{def:elliptic set}
    Let $A \in \Psitsc^{m, r}$.  The $\tsc$-elliptic set
    $\Elltsc(A)$ is
    \begin{align*}
        \Elltsc(A) = \Ell_{\fib}(A) \cup \Ell_{\mf}(A) \cup \Ell_{\ff}(A) \subset \Ctscd\,,
    \end{align*}
    with
    \begin{align*}
        \Ell_{\fib}(A) &= \{\alpha \in \Ssc_{X \setminus \poles}^* X \colon
                         \fibsymb{m,r}(A)(\alpha) \neq 0  \}\,,\\
        \Ell_{\mf}(A) &= \{\alpha \in \Tsco_{\pa X \setminus \poles}^* X \colon
                        \mfsymb{m,r}(A)(\alpha) \neq 0  \}\,,\\
    \end{align*}
    while
  \begin{align*}
        \Ell_{\ff}(A) &= \{\tau \in W^\perp \colon
                        \ffsymb{r}(A)(\tau) \text{ is scattering
                        elliptic and invertible}\}\, \\
    &\phantom{= \{} \cup \{ \pm \infty \in \partial \Wperpo :
      \fibsymb{m,r}(A) \text{ is nowhere vanishing on } UH_\pm \}.
    \end{align*}

    Moreover, we set
    \begin{align}\label{eq:def_Chartsc}
        \Chartsc(A) \coloneqq \Ctscd \setminus \Elltsc(A)\,,
    \end{align}
    the $\tsc$-characteristic set of $A$.
\end{definition}

As expected, one obtains elliptic estimates on the elliptic set.
\begin{proposition}[$\tsc$-elliptic regularity, cf.~\cite{BDGR}*{Prop.~4.20},
  \cite{Vasy2000}*{Lemma 9.3}]
  \label{prop:elliptic_regularity_tsc}
Let $u \in \schwartz'$ and $A \in \Psitsc^{m,r}$, $B, Q' \in
\Psitsc^{0,0}$.  Assume that $\WFtsc'(B) \subset \Elltsc(A) \cap
\Elltsc(Q')$.  For $s, \ell \in \mathbb{R}$, if $Q' A u \in \Sobsc^{s
  - m, \ell - r}$, then $Bu \in \Sobsc^{s,\ell}$ and for any $M,N \in
\mathbb{R}$ there is $C > 0$ such that
\begin{equation*}
  \norm{B u}_{s, \ell} \le C \left(\norm{Q' A u}_{s - m, \ell - r}  + \normres{u} \right).
\end{equation*}
\end{proposition}

In order to state a wavefront set condition for the Feynman propagator, we have to define the $\tsc$-wavefront set.
Our definition is very similar to the one in \cite{Vasy2000}, with the modification that we use $\Ctscd$.
We note that there is a related wavefront set defined in \cite{Vasy2001}, which is slightly weaker.

\begin{definition}
    Let $u \in \CdmI(X)$. A point $\zeta \in \Ctscd$ is not in the $\tsc$-wavefront set, $\zeta \not \in \WFtsc(u)$, if there exists $A \in \Psitsc^{0,0}$ with $\zeta \in \Elltsc(A)$
    such that $A u \in \CdI(X)$.
\end{definition}

For $\zeta \in \Ssc^*_{X \setminus \poles} X \cup \Tsc^*_{\pa X \setminus \poles} X$, the $\tsc$-wavefront set $\WFtsc(u)$ coincides with the normal (scattering-) wavefront set $\WF(u)$.
Over the poles, we have that if $\tau_0 \in W^\perp$ with $\WF(u) \cap \overline{(\pi^{\perp})^{-1}(\tau_0)} = \varnothing$, then $\tau_0 \not \in \WF_{\ff}(u)$,
where $\WF_{\ff}(u) \coloneqq \WFtsc(u) \cap \Wperpo$.

\section{\texorpdfstring{$\tsc$}{3sc}-localizers}\label{sec:tsc localizers}

As in the scattering case in Section~\ref{sec:scattering section}, we
define a domain $\cX$ for the operator $P_{V}$, which will become the
range of the Feynman propagator, and this domain is again defined
using microlocalizers to the radial sources analogous with those in
equation~\eqref{eq:scattering source cutoff}.  In this section we
describe these microlocalizers and record several necessary
modifications of estimates in our previous work~\cite{BDGR}.

\subsection{Localization to the characteristic set}
Following Vasy~\cite{Vasy2000}, we use the functional calculus as in our
previous work~\cite{BDGR}*{Section 5} to adapt the
microlocalizers to the radial sets over the poles.  

We first describe the functions of the operators  used in
the construction.
Let $V_{0} \in \Psisc^{1,-1}(\RR^n)$, $V_{0}^{*} = V_{0}$.  For $E > 0$ sufficiently large and $\psi \in \CcI(\RR)$, the operator
\begin{align*}
    G_{\psi, V_{0}} \coloneqq \psi\left( (D_t^2 + H_{V_0} + E)^{-1} P_{V_0} \right)
\end{align*}
is well-defined by the functional calculus of the (static) operator
$H_{V_{0}}$, $G_{\psi,V_0} \in \Psitsc^{0,0}$, and satisfies (cf.~\cite{BDGR}*{Sect.~5.1})
\begin{itemize}
    \item $\ffsymbz(G_{\psi, V_0})(\tau) = \psi\left( (\tau^2 + H_{V_0} + E)^{-1} (\tau^2 - H_{V_0}) \right)$,
    \item $\ffsymbz(G_{\psi, V_0}) \in \Psisclsc^{-\infty,0,0}$.
\end{itemize}

In \cite{BDGR}, we only used $G_{\psi, V_0}$ near the poles, now we need a global operator $G_\psi$ that coincides with $G_{\psi,V_+}$ near $\NP$ and
$G_{\psi,V_-}$ near $\SP$.

\begin{definition}\label{def:asymptotic_static}
    Let $V \in \rho_{\mf} \Psitsc^{1,0}$ and $r \in \ioo{0, \infty}$.
    We say that $V$ is asymptotically static of order $r$ at $\poles$, if there exist
    $V_{\pm} \in S_{\cl}^{-1}(\RR^n_z)$ and $\chi_{\pm} \in \CcI(X)$ such that
    \begin{enumerate}
        \item $\chi_+(\NP) = 1$ and $\chi_-(\SP) = 1$,
        \item $\chi_{\pm} \cdot (V - V_{\pm}) \in \Psitsc^{1, -r}$.
    \end{enumerate}
\end{definition}
We note that if $V$ is asymptotically static of order $r > 0$, then the static parts $V_{\pm}$ are uniquely determined.

For the definition of $G_\psi$, we choose $\chi_{\pm}$ as in the previous definition and with the additional property that $\chi_+ = 0$ in a neighborhood of $\SP$
and $\chi_- = 0$ in a neighborhood of $\NP$.
We set
\begin{align*}
    G_\psi \coloneqq \chi_+ G_{\psi, V_+} + \chi_- G_{\psi, V_-} + (1 - \chi_+ - \chi_-) G_{\psi,0}\,.
\end{align*}
We have that $G_\psi \in \Psitsc^{0,0}$ and since by \cite{BDGR}*{Eq. (5.8) and Eq. (5.9)} the $\mf$ and $\fib$-symbols of $G_{\psi,V_0}$ are independent of
$V_0$, we have that $G_\psi$ is defined independently of the choice of $\chi_\pm$ up to $\Psitsc^{-1,-1}$.
    
We can estimate $B u$ by $BG_\psi u$ and $G' P_V u$
if $G'$ is elliptic on the wavefront set of $B$:
\begin{proposition}[cf. \cite{BDGR}*{Proposition~5.14}]\label{prop:elliptic_Gpsi_estimate}
    Let $V \in \rho_{\mf} \Psitsc^{1,0}$ be asymptotically static of order $r$.
    Let $\varphi \in \CcI(\RR)$ with $\varphi|_{(-\eps, \eps)} \equiv 1$ for some $\eps > 0$,
    $B, G', B' \in \Psitsc^{0, 0}$
    such that
    \begin{align*}
        \WFtsc'(B) \subset \Elltsc(G') \cap \Elltsc(B')\,.
    \end{align*}
    For all $M, N \in \NN$ and $s, \ell \in \RR$ there exists $C > 0$ such that for all $u \in \Sobres$,
    \begin{align*}
        \norm{B u}_{s,\ell} \leq C \left( \norm{B G_\varphi u}_{s,\ell} + \norm{G' P_V u}_{s-2, \ell} + \norm{B' u}_{s-1,\ell-r} + \normres{u}\right)\,,
    \end{align*}
    provided that the right hand side is finite.
\end{proposition}
\begin{proof}
    Recall~\cite{BDGR}*{Proposition~5.14} that we have $E_{\psi, V_0} \in \Psitsc^{-2, 0}$ which satisfies $(\id - G_{\psi, V_0}) = E_{\psi, V_0} P_{V_0}$.
    We set
    \begin{align*}
        \tilde{E} &= \chi_+ E_{\psi, V_+} + \chi_- E_{\psi, V_-}
        \intertext{and}
        R &= \chi_+ E_{\psi, V_+} (P_{V_+} - P_V) + \chi_- E_{\psi,
            V_-} (P_{V_-} - P_V) \in \Psitsc^{-1,-r}
    \end{align*}
    Therefore, we have that
    \begin{align*}
        \id = G_\psi + \tilde{E} P_V + R + (1 - \chi_+ - \chi_-) (\id - G_{\psi,0})
    \end{align*}
    Since $\WFtsc'(B) \subset \Elltsc(G')$, we obtain
    \begin{align*}
        \norm{B \tilde{E} P_V u}_{s,\ell} \lesssim \norm{G' P_V u}_{s-2, \ell} + \normres{u},
    \end{align*}
    while $\WFtsc'(B) \subset \Elltsc(B')$ implies
    \begin{align*}
        \norm{B R u}_{s,\ell} \lesssim \norm{B' u}_{s-1, \ell-r} + \normres{u},
    \end{align*}
    and, since $P_{V}$ is elliptic on the microsupport of $(1 - \chi_+ - \chi_-) (\id - G_{\psi,0}) \in \Psisc^{0,0}$,
    \begin{align*}
        \norm{B (1 - \chi_+ - \chi_-) (\id - G_{\psi,0}) u}_{s,\ell} \lesssim \norm{G' P_V u}_{s-2, \ell} + \normres{u}.
    \end{align*}
    Combining these estimates gives the claimed inequality.
\end{proof}

Moreover, we have an elliptic estimate for $B G_\varphi u$ by $Q G_\psi u$
given that $Q$ is elliptic on the wavefront set of $B G_\varphi$ and the support of $\varphi$ is contained in the set $\set{\psi = 1}$.
\begin{proposition}\label{prop:elliptic_regularity_Gpsi}
    Let $B, Q \in \Psitsc^{m,r}$ and $\varphi, \psi \in \CcI(\RR)$ with $\varphi \psi = \varphi$ and $\varphi(0) = 1$.
    If $\WFtsc'(B G_{\varphi}) \subset \Elltsc(Q)$, then for any $N, M \in \NN$ there exists $C > 0$ such that for $u \in \Sobsc^{-N, -M}$,
    \begin{align*}
        \norm{B G_{\varphi} u} \leq C \left( \norm{Q G_\psi u} + \normres{u}\right)\,.
    \end{align*}
\end{proposition}
\begin{proof}
    Let $\tilde{\chi}_{\pm} \in \CcI(X)$ with $\supp \tilde{\chi}_{\pm} \subset \set{\chi_{\pm} \equiv 1}$, 
    $\tilde{\chi}_+ \equiv 1$ in a small neighborhood of $\NP$,
    and $\tilde{\chi}_- \equiv 1$ in a small neighborhood of $\SP$.
    Write $B = B \tilde{\chi}_+ + B \tilde{\chi}_- + B (1 - \tilde{\chi}_+ - \tilde{\chi}_-)$.
    We have that
    \begin{align*}
        B \tilde{\chi}_+ G_\varphi = B \tilde{\chi}_+ G_{\varphi, V_+} = B \tilde{\chi}_+ G_{\varphi, V_+} G_{\psi, V_+}\,.
    \end{align*}
    On the support of $B \tilde{\chi}_+G_{\varphi,V_{+}}$, $G_{\psi, V_+}$ is equal to to $G_\psi$ modulo $\Psisc^{-\infty, -\infty}$ and together with elliptic regularity, we have
    Then, we have that
    \begin{align*}
        \norm{B \tilde{\chi}_+ G_\varphi u} \lesssim \norm{Q G_\psi u} + \normres{u}\,.
    \end{align*}
    Similarly, we can estimate $B \tilde{\chi}_- G_\varphi u$.
    for the interior term, we notice that $B (1 - \tilde{\chi}_+ - \tilde{\chi}_-) G_{\varphi}$ is microsupported in the elliptic set of $Q G_\psi$ and therefore
    we can apply elliptic regularity directly.
\end{proof}

In order to localize to slices $\set{\tau = \tau_0}$, we want to
quantize symbols whose restrictions to $\poles$ are purely functions
of $\tau$.  
For this, we denote the fiber equator by
\begin{align*}
    \fibeq \coloneqq \pa \overline{\RR_{\tau, \zeta}^{n+1}} \cap \set{\tau = \tau_0}\,.
\end{align*}
Note that this set is independent of the choice of $\tau_0$.

We recall from \cite{BDGR}*{Proposition~5.15} that if $q \in
\CI([\Tsco^* X; \fibeq])$, using $\rho = 1/\tau$ as a rescaling
function, 
\begin{align*}
    \Op_L(x^{-\ell} \rho^{-s} q) G_\psi \in \Psitsc^{s,\ell}
\end{align*}
for any $\psi \in \CcI$ and $s, \ell \in \RR$.

We have the following variant of \cite{Vasy2000}*{Lemma~13.5}:
\begin{lemma}\label{lem:WF_QGpsi}
    Assume that $q |_{\NP} = f,$ where $f \in \CI(\Wperpo)$, i.e.,
    $q_{\NP}$ depends only on $\tau$.
    The operator wavefront set of $Q \coloneqq \Op_L(x^{-\ell} \rho^{-s} q) G_\psi$ satisfies
    \begin{align*}
        \WF_{\ff}'(Q) &\subset \esssupp(f) \cap \WF_{\ff}'(G_\psi)\,,\\
        \WF_{\mf}'(Q) &\subset \esssupp_{\mf}\left(q \psi\left( \frac{\tau^2 - (\abs{\zeta}^2 + m^2)}{\tau^2 + \abs{\zeta}^2 + m^2 + E} \right)\right) \,,\\
        \WF_{\fib}'(Q) &\subset \esssupp_{\fib}\left(q \psi\left( \frac{\tau^2 - \abs{\zeta}^2}{\tau^2 + \abs{\zeta}^2} \right)\right) \,.
    \end{align*}
\end{lemma}

\subsection{Localization to the radial set}

The radial set is naturally defined on $\Tsco^* X$, but the natural
phase space for $\tsc$-operators is $\Ctscd$.  
Over the poles, we therefore define the $\tsc$-radial set as subsets of $W^\perp$, namely
\begin{align*}
    \Rad_{\ff,{}\sources} &\coloneqq  \left( \NP \times \set{ - m } \right) \cup \left( \SP \times \set{ + m} \right) \subset W^\perp\,,\\
    \Rad_{\ff,{}\sinks} &\coloneqq \left( \NP \times \set{ + m } \right) \cup \left( \SP \times \set{ - m} \right) \subset W^\perp\,.
\end{align*}

The entire $\tsc$-radial set is given by
\begin{align*}
    \Radtsc \coloneqq \left(\Rad \cap \left( \Ssc^*_{X \setminus \poles} X \cup \Tsc^*_{\pa X  \setminus \poles} X \right)\right) \cup \left( \poles \times \set{ \pm m } \right)
\end{align*}
and
\begin{align*}
    \Radtsc_{\sources} &\coloneqq \left(\Rad_{\sources} \cap \left( \Ssc^*_{X \setminus \poles} X \cup \Tsc^*_{\pa X  \setminus \poles} X \right) \right)\cup \Rad_{\ff,\sources}\,,\\
    \Radtsc_{\sinks} &\coloneqq \left(\Rad_{\sinks} \cap \left( \Ssc^*_{X \setminus \poles} X \cup \Tsc^*_{\pa X  \setminus \poles} X \right) \right)\cup \Rad_{\ff, \sinks}\,.
\end{align*}
We have that $\Radtsc = \Radtsc_{\sources} \sqcup \Radtsc_{\sinks}$
and $\Radtsc_{\bullet} \subset \gamma_{\tsc}( \Rad_{\bullet} )$ for $\bullet \in \set{\sources, \sinks}$.

We note that localizing to the characteristic set over the poles is
nuanced:  if $A \in \Psitsc^{0,0}$ is elliptic at $\Rad_{\ff,\sources}$, then there exists a neighborhood $U \subset \pa X$ of $\poles$ such that
$A$ is elliptic at $\Tsc^*_{U \setminus \poles} X \cap
\Rad_{\sources}$.  We therefore must further localize and employ a localizer of the form $Q G_\psi$ for appropriate
$\psi$ and $Q = \Op_L(q)$, with the symbol $q$ restricting over $\NP$
to a function of $\tau$.

\begin{definition}\label{def:localizer-to-src}
    Let $\delta > 0$.
    We call $Q \in \Psitsc^{0,0}$ a $\delta$-localizer to $\Radtsc_{\sources}$ if
    \begin{enumerate}
        \item $Q$ is microlocally the identity near the sources,
            \begin{align*}
                \WFtsc'(\id - Q) \cap \Radtsc_{\sources} = \varnothing\,,
            \end{align*}
        \item \label{it:Psitscsrc2} for all $\varphi \in \CcI(\RR)$ with $\supp \varphi \subset \ioo{-\delta,\delta}$
            we have
            \begin{align*}
                \WFtsc'(Q G_\varphi) \cap \Radtsc_{\sinks} = \varnothing\,.
            \end{align*}
    \end{enumerate}

    We denote the set of all $\delta$-localizers to
    $\Radtsc_{\sources}$ by $\Psitsc_{\sources,\delta}$.  We also define $\Psitsc_{\sinks, \delta}$, where the roles of $\Rad_{\sources}$ and $\Rad_{\sinks}$ are interchanged.
\end{definition}

\begin{lemma}\label{lem:there-are-localizers}
    The set $\Psitsc_{\sources, \delta}$ is non-empty for all $\delta > 0$.
\end{lemma}
\begin{proof}
    Let $\eps > 0$ small and choose $f \in \CcI(\Wperpo)$ such that
    $f(+\infty, \tau) = f(-\infty, -\tau) = 0$
    and $f(+\infty, -\tau) = f(-\infty, \tau) = 1$
    for $\tau \in (m - \eps, m + \eps)$.
    We choose a function $q \in \CI([\Tsco^* X;\fibeq])$ such that $q|_{\poles}(\pm \infty, \tau, \zeta) = f(\pm \infty, \tau)$ and
    $q = 1$ in a neighborhood of $\Rad_{\sources}$ and $q = 0$ in a neighborhood of $\Rad_{\sinks}$.

    Fix $\psi \in \CcI(\RR)$ with $\psi(s) \equiv 1$ for $s \in \ioo{-\delta,\delta}$ and $\psi(s) \equiv 0$ for $\abs{s} > 2\delta$.
    By \cite{BDGR}*{Proposition~5.15}, we have that
    \begin{align*}
        Q \coloneqq \Op_L(q) G_\psi + (\id - G_\psi)
    \end{align*}
    is a $\tsc$-operator of order $(0,0)$.

    To show that $\WFtsc'(\id - Q) \cap \Radtsc_{\sources} = \varnothing$, we use that
    \begin{align*}
        \id - Q = \Op_L(1 - q) G_\psi \,.
    \end{align*}
    Since $q = 1$ in a neighborhood of $\Rad_{\sources}$ and $f = 1$ near $\tau = -m$, we have that $\WFtsc'(\id - Q) \cap \Radtsc_{\sources} = \varnothing$
    by Lemma~\ref{lem:WF_QGpsi}.

    For $\varphi \in \CcI$ with $\supp \varphi \subset \ioo{-\delta, \delta}$, we calculate
    \begin{align*}
        Q G_\varphi = \Op_L(q) G_\varphi
    \end{align*}
    and consequently,
    \begin{align*}
        \WF_{\ff}'(Q G_\varphi) \subset \esssupp(f) \cap \WF_{\ff}'(G_\varphi)\,.
    \end{align*}
    Since $f(+\infty, m) = 0$ it then follows that $m \not \in \WF_{\ff}'(Q G_\varphi)$.

    The same arguments apply to $\SP$ and therefore $Q \in
    \Psitsc_{\sources, \delta}$.
\end{proof}

\subsection{Propagation estimates}\label{sec:tsc estimates}

\begin{definition}
    Let $U, V, W \subset \Ctscd$. We say that
    $U$ is (backward) controlled by $V$ through $W$ if
    for all $\alpha \in \Char(P_0)$ that are incoming to $U$, in sense that
    \begin{align*}
        \pi_{X, \tau}(\alpha) \in \pi_{X,\tau}( \gamma_{\tsc}(U) \cap \Char(P_0))\,,
    \end{align*}
    there exists $s_{\alpha} <0$ such that
    \begin{align*}
        \exp(s_{\alpha} \Hamsc_p)(\alpha) \in V
    \end{align*}
    and for all $s \in [s_{\alpha}, 0]$,
    \begin{align*}
        \exp(s \Hamsc_p)(\alpha) \in W\,.
    \end{align*}
\end{definition}

We recall from \cite{BDGR} the various propagation estimates. We assume throughout that
$V \in \rho_{\mf} \Psitsc^{1,0}$ is asymptotically static of order $r \geq 1$ and
\begin{align*}
    V - V^* \in \Psitsc^{0,-2}\,.
\end{align*}
Moreover, we assume that $H_{V_{\pm}}$ have purely absolutely
continuous spectrum in $[m^2, \infty)$.  This condition is guaranteed
on $(m^{2},\infty)$ by well-known results in scattering theory; it is
really only an assumption at the bottom of the continuous spectrum $m^{2}$.

\begin{proposition}[Regular propagation estimate~\cite{BDGR}*{Proposition~6.2}]\label{prop:propagation_localized}
    Let $\delta > 0$ sufficiently small,
    $\varphi, \psi_1, \psi_2 \in \CcI(\RR)$ with $\supp \varphi \subset \ioo{-\delta, \delta}$ and $\psi_j|_{\ioo{-\delta, \delta}} \equiv 1$,
    and $B, E, G, G', B' \in \Psitsc^{0,0}$ such that
    \begin{enumerate}
        \item $\WF_{\ff}'(E) = \varnothing$,
        \item $\Hamsc_p(\alpha) \not = 0$ for all $\alpha \in \gamma_{\tsc}(\Elltsc(G))$,
        \item $\WFtsc'(B G_\varphi)$ is controlled by $\Elltsc(E)$ through $\Elltsc(G)$,
        \item $\WFtsc'(B) \subset \Elltsc(G') \cap \Elltsc(B')$.
    \end{enumerate}
    and $\psi_1, \psi_2 \in \CcI(\RR)$ with $\psi_j|_{(-\delta, \delta)} \equiv 1$.

    For all $M, N, s, \ell \in \RR$ and $u \in \Sobres$ with
    $E G_{\psi_1} u \in \Sobsc^{s,\ell}$,
    $G G_{\psi_2} P_V u \in \Sobsc^{s-1, \ell+1}$,
    $G' P_V u \in \Sobsc^{s-2, \ell}$,
    and $B' u \in \Sobsc^{s-1, \ell-r}$,
    it follows that $Bu \in \Sobsc^{s,\ell}$ and
    \begin{align*}
        \norm{B u}_{s,\ell} &\leq C \big( \norm{E G_{\psi_1} u}_{s,\ell} + \norm{G G_{\psi_2} P_V u}_{s-1, \ell+1} + \norm{G' P_V u}_{s-2, \ell} + \norm{B' u}_{s-1, \ell-r} \\
        &\phantom{\leq C \big(} + \normres{u} \big)\,.
    \end{align*}
\end{proposition}

Near the radial sets $\Radtsc$, we have two different estimates, depending on whether $\ell > 1/2$ or $\ell < - 1/2$:

\begin{proposition}[Above threshold radial point estimate~\cite{BDGR}*{Proposition~7.1}]
    \label{prop:tsc_radial_above}
    Let $\delta > 0$ sufficiently small, $\varphi, \psi_1, \psi_2 \in \CcI(\RR)$ with $\supp \varphi \subset \ioo{-\delta, \delta}$ and $\psi_j|_{\ioo{-\delta, \delta}} \equiv 1$,
    $\tau_0 \in \set{\pm m} \subset W^\perp$,
    and $B, B_1, G, G', B' \in \Psitsc^{0,0}$ such that
    \begin{enumerate}
        \item $\WFtsc'(B G_{\varphi})$ is contained in a sufficiently small neighborhood of $\tau_0 \in \Ctscd$,
        \item $\tau_0 \in \Ell_{\ff}(B) \cap \Ell_{\ff}(B_1)$,
        \item $\WFtsc'(B G_{\varphi}) \subset \Elltsc(B_1) \cap \Elltsc(G)$,
        \item $\WFtsc'(B) \subset \Elltsc(B') \cap \Elltsc(G')$.
    \end{enumerate}

    For all $M, N, s, s', \ell, \ell' \in \RR$ with $\ell > \ell' > -1/2$, $s > s'$,
    and $u \in \Sobres$ with
    $B_1 G_{\phi} u \in \Sobsc^{s', \ell'}$,
    $G G_{\psi} P_V u \in \Sobsc^{s-1, \ell+1}$,
    $G' P_V u \in \Sobsc^{s-2, \ell}$,
    and $B' u \in \Sobsc^{s- 1, \ell - r}$,
    it follows that $Bu \in \Sobsc^{s, \ell}$ and
    \begin{align*}
        \norm{B u}_{s,\ell} &\leq C \big( \norm{B_1 G_{\psi_1} u}_{s',\ell'} + \norm{G G_{\psi_2} P_V u}_{s-1, \ell+1} + \norm{G' P_V u}_{s-2, \ell} + \norm{B' u}_{s-1, \ell-r} \\
        &\phantom{\leq C \big(} + \normres{u} \big)\,.
    \end{align*}
\end{proposition}

\begin{proposition}[Below threshold radial point estimate~\cite{BDGR}*{Proposition~7.2}]
    \label{prop:tsc_radial_below}
    Let $\delta > 0$ sufficiently small, $\varphi, \psi_1, \psi_2 \in \CcI(\RR)$ with $\supp \varphi \subset \ioo{-\delta, \delta}$ and $\psi_j|_{\ioo{-\delta, \delta}} \equiv 1$,
    and $B, E, G, G', B' \in \Psitsc^{0,0}$ such that

    \begin{enumerate}
        \item $\WFtsc'_{\ff}(E) = \varnothing$,
        \item $\WFtsc'(B G_{\varphi}) \cup \Elltsc(E) \subset \Elltsc(G)$,
        \item $\Elltsc(G) \cap \Radtsc_{\sources} = \varnothing$,
        \item $\WFtsc'(B) \subset \Elltsc(B') \cap \Elltsc(G')$,
        \item $\WFtsc'(B G_{\varphi}) \setminus \Radtsc_{\sinks}$ is backward controlled by $\Elltsc(E)$ through $\Elltsc(G)$.
    \end{enumerate}

    For all $M, N, s, \ell \in \RR$ with $\ell < -1/2$ and $u \in \Sobres$ with
    $E G_{\psi_1} u \in \Sobsc^{s,\ell}$,
    $G G_{\psi_2} P_V u \in \Sobsc^{s-1, \ell+1}$,
    $G' P_V u \in \Sobsc^{s-2, \ell}$,
    and $B' u \in \Sobsc^{s-1, \ell-r}$,
    it follows that $B u \in \Sobsc^{s,\ell}$ and
    \begin{align*}
        \norm{B u}_{s,\ell} &\leq C \big( \norm{E G_{\psi_1} u}_{s,\ell} + \norm{G G_{\psi_2} P_V u}_{s-1, \ell+1} + \norm{G' P_V u}_{s-2, \ell} + \norm{B' u}_{s-1, \ell-r} \\
        &\phantom{\leq C \big(} + \normres{u} \big)\,.
    \end{align*}

    The same statement holds if the roles of $\Radtsc_{\sources}$ and $\Radtsc_{\sinks}$ are interchanged and forward control is used instead of backward control.
\end{proposition}

\section{Construction of the Feynman propagator}\label{sec:tsc prop construction}

We now construct the Feynman propagator.  We follow the general
structure of the construction in Section \ref{sec:scattering section},
which is to say that we begin with the construction of a Fredholm
problem for $P_V$.  The spaces involved in this construction are
similar to those in Section \ref{sec:scattering section} in the sense
that they have an overall regularity and below threshold weight with
the assumption of above threshold weight imposed at the radial sources
using a microlocalizer, the primary difference here being that we must now use
the $\tsc$-microlocalizers discussed above.

We choose $\delta > 0$,
$Q_{\sources} \in \Psitsc_{\sources,\delta}$,
$Q_{\sources}' \in \Psitsc_{\sources, 2 \delta}$ and cutoff functions $\phi, \psi \in \CcI(\RR)$ be bump functions supported near
$0$ with $\supp \phi \subset (-\delta, \delta)$, $\supp \psi \subset
(-2 \delta, 2 \delta)$, and $\psi(s) \equiv 1$
on $(-\delta, \delta)$.  ($\Psitsc_{\sources,\delta}$ is from
Definition \ref{def:localizer-to-src}.) We assume
\begin{equation*}
\WFtsc'(Q_{\sources} G_{\phi}) \subset \Elltsc(Q_{\sources}').
\end{equation*}
Moreover, we require, as in the
scattering case in Section \ref{sec:microlocal cutoffs scattering},
that the segments of broken bicharacteristic rays with endpoints in
$\WFtsc'(Q_{\sources} G_{\phi})$ lie in $\Elltsc(Q_{\sources}')$.

We additionally choose a $B_{\sources} \in \Psitsc^{0,0}$ with $\Rad_{\ff,\sources} \subset \Ell_{\ff}(B_{\sources})$.
For $s, \ell_0, \ell_+ \in \RR$, we define the Feynman--Sobolev spaces as
\begin{align*}
    \cY^{s,\ell_0,\ell_+} &\coloneqq \set*{ v \in \Sobsc^{s-1, \ell_0+1} \colon Q_{\sources}' G_\psi v \in \Sobsc^{s-1, \ell_+ + 1}, B_{\sources} v \in \Sobsc^{s-2, \ell_+}}\,,\\
    \cX^{s,\ell_0,\ell_+} &\coloneqq \set*{ u \in \Sobsc^{s,\ell_0} \colon Q_{\sources} G_\phi u \in \Sobsc^{s,\ell_+}, P_V u \in \cY^{s,\ell_0, \ell_+}}
\end{align*}
with norms
\begin{align*}
    \norm{v}_{\cY^{s,\ell_0,\ell_+}}^2 &\coloneqq \norm{v}^2_{s-1,\ell_0 + 1} + \norm{Q_{\sources}' G_\psi v}^2_{s - 1,\ell_+ + 1} + \norm{B_{\sources} v}^2_{s-2, \ell_+}\,,\\
    \norm{u}_{\cX^{s,\ell_0, \ell_+}}^2 &\coloneqq \norm{u}_{s,\ell_0}^2 + \norm{Q_{\sources} G_\phi}_{s,\ell_+}^2 + \norm{P_V u}_{\cY^{s,\ell_0, \ell_+}}^2\,.
\end{align*}

In particular, this means that $u$ is supposed to have above threshold regularity at $\Rad_{\sources}$ and below threshold regularity at $\Rad_{\sinks}$.
We have to include the $B_{\sources} v$ term in the
$\cY^{s,\ell_0,\ell_+}$ space because the above threshold radial set estimate near the poles
requires control of $B_{\sources} P_V u$ with a $\tsc$-elliptic
operator and the operator $Q_{\sources} G_\psi$ is not $\tsc$-elliptic there.
If we interchange the roles of $\Rad_{\sources}$ and $\Rad_{\sinks}$, then the resulting spaces give the Fredholm problem associated to the anti-Feynman propagator.

As in \cite{BDGR} we have two results, which differ by the assumption
on bound states and the decay of the non-static part of the
potential.  We treat the case where there are no bound states for the
limiting Hamiltonians in this section and the case with bound states
in Section \ref{sec:bound states}.

\begin{theorem}\label{thm:invertible_no_bound}
    Let $s, \ell_0, \ell_+, r \in \RR$ with $\ell_0 < -1/2 < \ell_+<1/2$, and $r \geq \max\set{1, \ell_+ - \ell_0}$.
    Let $V \in \rho_{\mf} \Psitsc^{1,0}$ be an asymptotically static perturbation of order $r$ and
    the limiting Hamiltonians $H_{V_{\pm}} = \Delta + m^2 + V_{\pm}$ have purely absolutely continuous spectrum near $[m^2, \infty)$, have no bound states, and
    the leading part of $V$ is self-adjoint in the sense that
    \begin{align*}
        V - V^* \in \Psitsc^{0,-2}\,,
    \end{align*}
    then the map
    \begin{align}\label{eq:fred map tsc}
        P_V : \cX^{s,\ell_0, \ell_+} \to \cY^{s,\ell_0, \ell_+}
    \end{align}
    is Fredholm.

    If we assume that $V \in \rho_{\mf} \Difftsc^{1}$ and $V = V^*$, then $P_V$ is invertible.
\end{theorem}

To prove that $P_V$ is a Fredholm operator, we have to show that $P_V : \cX^{s,\ell_0,\ell_+} \to \cY^{s,\ell_0, \ell_+}$ has finite dimensional kernel and cokernel.

\begin{lemma}\label{lem:finite-dim_kernel}
    Let $s, \ell_0, \ell_+, r \in \RR$ with $\ell_0 <-1/2<\ell_{+}$, and $r \geq \max\set{1, \ell_+ - \ell_0}$.
    Assume that $V \in \rho_{\mf} \Psitsc^{1,0}$ satisfies the assumptions of Theorem~\ref{thm:invertible_no_bound}.
    There exists $C > 0$ such that for all $u \in \Sobsc^{s-1, \ell_0 - 1} \cap \cX^{s,\ell_0, \ell_+}$,
    \begin{align*}
        \norm{u}_{\cX^{s,\ell_0,\ell_+}} \leq C \left( \norm{P_V u}_{\cY^{s,\ell_0, \ell_+}} + \norm{u}_{s-1, \ell_0 - 1}\right)\,.
    \end{align*}
    In particular, $\ker_{\cX^{s,\ell_0, \ell_+}} P_V$ is
    finite-dimensional and has closed range.
\end{lemma}
\begin{proof}
    The claimed global estimate follows, using a interpolation argument as in \cite{BDGR}*{Eq.~(2.39)} applied to $\norm{Q_{\sources} G_{\phi} u}_{s,\ell_0}$, from the two estimates
    \begin{align*}
        \norm{u}_{s,\ell_0} &\lesssim \norm{Q_{\sources} G_\phi u}_{s,\ell_0} + \norm{P_V u}_{s-1, \ell_0 + 1} + \norm{u}_{s-1, \ell_0 - 1} \\
        \norm{Q_{\sources} G_\phi u}_{s,\ell_+} &\lesssim \norm{B_{\sources} P_V u}_{s-2, \ell_+} + \norm{Q_{\sources}' G_\psi P_V u}_{s-1,\ell_+ + 1} \\
        &\phantom{\lesssim} + \norm{u}_{s, \ell_0}\,.
    \end{align*}

    The first inequality is the result of using a microlocal partition of unity together with elliptic estimates, propagation estimates and below threshold estimates.
    Near the poles, we use Proposition~\ref{prop:elliptic_regularity_tsc}, Proposition~\ref{prop:propagation_localized}, and Proposition~\ref{prop:tsc_radial_below} and away from the poles, we can use the estimates
    from the scattering calculus as in Section~\ref{sec:scattering section}.
    
    More precisely, we take an open cover $O_1, O_2, O_3, O_4$ of the compressed cotangent
    bundle $\Tdoto X$ similar as in \cite{BDGR}*{p.~83} such that 
    \begin{enumerate}
        \item $\WFtsc'(Q_{\sources}G_{\phi}) \subset O_1$,
        \item $\Rad_{\sinks} \subset O_2 \subset \WFtsc'(\id - Q_{\sources})$,
        \item $\Chartsc(P_V) \subset O_1 \cup O_2 \cup O_3$,
        \item $O_3 \cap \Chartsc(P_V)$ is controlled along $\Hamsc_p$ by $O_1$,
        \item $(O_2 \cap \Chartsc(P_V)) \setminus \Rad_{\sinks}$ is controlled along $\Hamsc_p$ by $O_3$, and
        \item $O_4 \subset \Elltsc(P_V)$.
    \end{enumerate}
    The proof that that such a cover exists is similar to the case in \cite{BDGR}.

    We also choose a collection $Q_2, Q_3, Q_4 \in \Psitsc^{0,0}$ with
    \begin{align*}
        \WFtsc'(Q_i G_{\phi}) \subset O_i
    \end{align*}
    and
    \begin{align*}
        \Ctscd \subset \Elltsc(Q_{\sources}) \cup \bigcup_{i = 2}^4 \Elltsc(Q_i)\,.
    \end{align*}

    We have the estimate
    \begin{align*}
        \norm{u}_{s,\ell_0} \lesssim \norm{Q_{\sources} u}_{s,\ell_{0}} + \norm{Q_2 u}_{s,\ell_{0}} + \norm{Q_3 u}_{s,\ell_{0}} + \norm{Q_4 u}_{s,\ell_{0}} + \norm{u}_{s-1, \ell_0-1}\,.
    \end{align*}

    By Proposition~\ref{prop:elliptic_Gpsi_estimate}, we have
    \begin{align*}
        \norm{Q_{\sources} u}_{s,\ell_0} \lesssim \norm{Q_{\sources} G_\phi u}_{s,\ell_0} + \norm{P_V u}_{s-1, \ell_0 + 1} + \norm{u}_{s-1, \ell_0 - 1}\,.
    \end{align*}
    Using the propagation estimates, Proposition~\ref{prop:propagation_localized} and \cite{BDGR}*{Proposition~2.6}, we can estimate $Q_3 u$ by the localizer near the sources, $Q_{\sources} u$, and $P_V u$.
    More precisely, we have
    \begin{align*}
        \norm{Q_3 u}_{s,\ell_0} \lesssim \norm{Q_{\sources} u}_{s,\ell_0} + \norm{P_V u}_{s-1, \ell_0 + 1} + \norm{u}_{s-1, \ell_0 -1}\,.
    \end{align*}
    The below threshold estimates, Proposition~\ref{prop:tsc_radial_below} and \cite{BDGR}*{Proposition~2.13}, imply
    \begin{align*}
        \norm{Q_2 u}_{s,\ell_0} \lesssim \norm{Q_3 u}_{s,\ell_0} + \norm{P_V u}_{s-1, \ell_0 + 1} + \norm{u}_{s-1, \ell_0 - 1}\,.
    \end{align*}
    Lastly, we use the elliptic estimate \cite{BDGR}*{Proposition~2.2} to bound
    \begin{align*}
        \norm{Q_4 u}_{s,\ell_0} \lesssim \norm{P_V u}_{s-2, \ell_0} + \normres{u} \leq \norm{P_V u}_{s-1, \ell_0 + 1} + \norm{u}_{s-1, \ell_0 - 1}\,.
    \end{align*}
    Putting the estimates together, we obtain
    \begin{align*}
        \norm{u}_{s,\ell_0} \lesssim \norm{Q_{\sources} G_\phi u}_{s, \ell_0} + \norm{P_V u}_{s-1, \ell_0 + 1} + \norm{u}_{s-1, \ell_0 - 1}\,,
    \end{align*}
    which is the first inequality.

    The second inequality is a ``global'' above threshold estimate.
    As a consequence of Proposition~\ref{prop:tsc_radial_above} with $s' = s-1$, $-N = s-1$, and $-M = \ell_0$ and using that $\ell_+ - r \leq \ell_0$, we obtain
    \begin{equation}\label{eq:Qsources_weak}
    \begin{aligned}
        \norm{Q_{\sources} G_\phi u}_{s,\ell_+} &\lesssim \norm{B_1 G_\psi u}_{s-1, \ell'} + \norm{Q_{\sources}' G_\psi P_V u}_{s-1,\ell_+ + 1} + \norm{B_{\sources} P_V u}_{s-2, \ell_+} \\
        &\phantom{\lesssim}
        + \norm{u}_{s-1, \ell_0}\,,
    \end{aligned}
    \end{equation}
    where $\ell' \in (-1/2, \ell_+)$ and for some $B_1 \in \Psitsc^{0,0}$ that satisfies $\WFtsc'(Q_{\sources} G_\phi) \subset \Elltsc(B_1)$ and
    $\WFtsc'(B_1 G_\psi) \subset \Elltsc(B_{\sources}) \cup \Elltsc(Q_{\sources}' G_\psi)$.

    To remove the $B_1 G_\psi$ term, we claim that
    \begin{equation}\label{eq:tildeB_estimate}
    \begin{aligned}
        \norm{B_1 G_\psi u}_{s-1, \ell'} &\lesssim \norm{Q_{\sources} G_\phi u}_{s-1,\ell'} + \norm{B_{\sources} P_V u}_{s-3,\ell'} + \norm{Q_{\sources}' G_\psi P_V u}_{s-2, \ell' + 1} \\
        &\phantom{\lesssim} + \norm{u}_{s-2, \ell_0}\,.
    \end{aligned}
    \end{equation}
    Combining \eqref{eq:Qsources_weak} and \eqref{eq:tildeB_estimate}, we arrive at the estimate
    \begin{align*}
        \norm{Q_{\sources} G_\phi u}_{s,\ell_+} &\lesssim \norm{B_{\sources} P_V u}_{s-2, \ell_+} + \norm{Q_{\sources}' G_\psi P_V u}_{s-1,\ell_+ + 1} \\ 
        &\phantom{\lesssim} + \norm{Q_{\sources} G_\phi u}_{s-1,\ell'} + \norm{u}_{s, \ell_0}\,.
    \end{align*}
    Using the interpolation inequality \cite{BDGR}*{Eq.~(2.39)}, we can absorb $\norm{Q_{\sources} G_\phi u}_{s-1,\ell'}$ into the left hand side.

    It remains to prove \eqref{eq:tildeB_estimate}.
    From Proposition~\ref{prop:elliptic_Gpsi_estimate} we obtain
    \begin{align*}
        \norm{B_1 G_\psi u}_{s-1, \ell'} \lesssim \norm{B_1 G_\varphi u}_{s-1, \ell'} + \norm{B_{\sources} P_V u}_{s-3, \ell'} + \norm{u}_{s-1, \ell_0}\,,
    \end{align*}
    where $\varphi \in \CcI(\RR)$ with $\varphi \phi = \varphi$ and $\varphi \equiv 1$ in a small neighborhood of $0$.
    We choose $\tilde{Q} \in \Psitsc^{0,0}$ such that $\WFtsc'(\id - \tilde{Q}) \cap \Radtsc_{\sources} = \varnothing$ and $\WFtsc'(\tilde{Q} G_{\varphi}) \subset \Elltsc(Q_{\sources})$.
    Using Proposition~\ref{prop:elliptic_regularity_Gpsi}, we further estimate
    \begin{align*}
        \norm{B_1 G_\varphi u}_{s-1, \ell'} &\lesssim \norm{ B_1 (\id - \tilde{Q}) G_\varphi u}_{s-1, \ell'} + \norm{Q_{\sources} G_\phi u}_{s-1,\ell'} \\
        &\phantom{\lesssim} + \norm{u}_{s-1, \ell_0}\,.
    \end{align*}
    Since $B_1 (\id - \tilde{Q}) G_\varphi$ is localized away from the radial sources and controlled by $Q_{\sources} G_\phi$, we obtain using Proposition~\ref{prop:propagation_localized},
    \begin{align*}
        \norm{ B_1 (\id - \tilde{Q}) G_\varphi u}_{s-1, \ell'} &\lesssim \norm{Q_{\sources} G_\phi u}_{s-1, \ell'} + \norm{Q_{\sources}' G_\psi P_V u}_{s-2, \ell' + 1} \\
        &\phantom{\lesssim} + \norm{B_{\sources} P_V}_{s-3, \ell'} + \norm{u}_{s-2, \ell_0}\,.
    \end{align*}
    Putting all of these estimates together, we arrive at \eqref{eq:tildeB_estimate}.

    The conclusion that $\ker_{\cX^{s,\ell_0, \ell_+}} P_V$ is finite-dimensional and that $P_V$ has closed range follow from a standard argument (cf. \cite{Vasy2013-2})
    because $\Sobsc^{s-1, \ell_0 - 1} \to \cX^{s,\ell_0, \ell_+}$ is compact.
\end{proof}

To complete the proof of Theorem~\ref{thm:invertible_no_bound}, we have to show that the cokernel of $P_V : \cX^{s,\ell_0,\ell_+} \to \cY^{s,\ell_0, \ell_+}$ is finite-dimensional as well.
The cokernel is given by
\begin{align*}
    \coker(P_V) = \set*{v \in (\cY^{s,\ell_0,\ell_+})' \colon \ang{P_V u, v} = 0 \text{ for all } u \in \cX^{s,\ell_0,\ell_+}}\,,
\end{align*}
which is equal to $\ker_{(\cY^{s,\ell_0,\ell_+})'}P_V^*$ since $P_V$
has closed range.

\begin{lemma}\label{lem:localizer-mapping}
    There exists $\tilde{Q} \in \Psitsc_{\sinks,\delta}$ such that
    for all $\phi' \in \CcI(\RR)$ with $\supp \phi' \subset (-\delta, \delta)$ and
    $u \in (\cY^{s,\ell_0,\ell_+})'$, we have $\tilde{Q} G_{\phi'} u \in \Sobsc^{1-s, \ell_+'}$,
    where $\ell_+' = \min\set{-\ell_0-1, -\ell_+}$.
\end{lemma}
\begin{proof}
    Since $Q_{\sources}' \in \Psitsc_{\sources, \delta}$, we can construct the following operator $\tilde{Q}$:

    We choose $f \in \CcI(W^{\perp})$ such that 
    \begin{align*}
        f|_{\WF_{\ff}'(Q_{\sources}')^c} &= 0\,,\\
        f \rvert_{W^{\perp}_{+}}(\tau) &= f \rvert_{W^{\perp}_{-}} (-\tau) = 1
    \end{align*}
    for $\tau \in (m - \eps, m + \eps)$ and $\eps > 0$ sufficiently small.
    Moreover, we choose $\tilde{q} \in \CI([\Tsco^*X;\fibeq])$ such that
    $\tilde{q} = 1$ near $\Rad_{\sinks}$, $\tilde{q} = 0$ on $\WFtsc'(Q_{\sources}) \setminus \Wperpo$ and
    $\tilde{q}|_{\poles}(\tau, \zeta) = f(\tau)$.

    We set
    \begin{align*}
        \tilde{Q} \coloneqq \Op_L(\tilde{q}) G_\varphi + (\id - G_{\varphi})\,,
    \end{align*}
    which is an element in $\Psitsc_{\sinks, \delta}$ by construction.

    To show that $\tilde{Q} G_{\phi'}$ maps $(\cY^{s,\ell_0,\ell_+})'$ to $\Sobsc^{1-s,\ell_+'}$, we show that the adjoint $G_{\phi'} \tilde{Q}^*$ maps
    $\Sobsc^{s-1, -\ell_+'}$ to $\cY^{s,\ell_0,\ell_+}$.  
    Let $v \in \Sobsc^{s-1, -\ell_+'}$, we have to show that
    \begin{align*}
        G_{\phi'} \tilde{Q}^* v &\in \Sobsc^{s-1, \ell_0+1}\,,\\
        Q_{\sources}' G_\psi G_{\phi'} \tilde{Q}^* v &\in \Sobsc^{s-1, \ell_+ + 1}\,,\\
        B_{\sources} G_{\phi'} \tilde{Q}^* v &\in \Sobsc^{s-2, \ell_+}\,.
    \end{align*}
    Since $-\ell_+' \geq \ell_0 + 1$ and $-\ell_+' \geq \ell_+$, the first and the third property are trivially satisfied.
    Finally, we have that 
    \begin{align*}
        \WFtsc'(\tilde{Q} G_{\phi'}) \cap \WFtsc'(Q_{\sources}' G_\psi) = \varnothing
    \end{align*}
    by construction of $\tilde{Q}$ and therefore $Q_{\sources}' G_\psi G_{\phi'} \tilde{Q}^*$ is regularizing.
\end{proof}

\begin{lemma}\label{lem:finite-dim_cokernel}
    Let $s, \ell_0, \ell_+, r \in \RR$ with $\ell_0 < -1/2 < \ell_{+} < 1/2$, and $r \geq \max\set{1, \ell_+ - \ell_0}$.
    If $V \in \rho_{\mf} \Psitsc^{1,0}$ is an admissible asymptotically static perturbation of order $r$, then the kernel of $P_V^*$, $\ker_{(\cY^{s,\ell_0,\ell_+})'}P_V^*$, is finite-dimensional.
\end{lemma}
\begin{proof}
    Set $\ell_+' \coloneqq \min\set{-\ell_0 - 1, - \ell_+}$ and $\ell_0' = -\ell_+ - 1$. By assumption, $\ell_+' > -1/2$ and $\ell_0' < -1/2$
    and we note that $r \geq \ell_+' - \ell_0'$. Therefore, $s' = 1-s, \ell_0', \ell_+'$, and $r$ satisfy the assumptions of Lemma~\ref{lem:finite-dim_kernel}.

    We choose $\tilde{Q} \in \Psitsc_{\sinks,\delta}$ as in the previous lemma,
    $\tilde{Q}' \in \Psitsc_{\sinks, 2\delta}$ with $\WFtsc'(\tilde{Q}) \subset \Elltsc(\tilde{Q}')$ and as in Lemma~\ref{lem:finite-dim_kernel}.
    Moreover, we choose a $\tilde{B}_{\sinks} \in \Psitsc^{0,0}$ with $\Rad_{\ff,\sinks} \subset \Elltsc(\tilde{B}_{\sinks})$.
    We set
    \begin{align*}
        \tilde{\cY} &\coloneqq \set*{v \in \Sobsc^{-s,\ell_0'+1} \colon \tilde{Q}' G_{\psi} v \in \Sobsc^{-s, \ell_+' + 1}, \tilde{B}_{\sinks} v \in \Sobsc^{-s-1,-\ell_+'} }\,,\\
        \tilde{\cX} &\coloneqq \set*{u \in \Sobsc^{1-s,\ell_0'} \colon \tilde{Q} G_\phi u \in \Sobsc^{1-s,\ell_+'}, P_V^* u \in \tilde{\cY}}
    \end{align*}
    We claim that $\ker_{(\cY^{s,\ell_0,\ell_+})'}P_V^* \subset \tilde{\cX}$.
    We have to show that for $u \in (\cY^{s,\ell_0,\ell_+})'$ with $P_V^* u = 0$, we have that
    \begin{align*}
        u &\in \Sobsc^{1-s,\ell_0'}\,,\\
        \tilde{Q} G_\phi u &\in \Sobsc^{1-s,\ell_+'}\,.
    \end{align*}
    The first inclusion follows by duality from $\Sobsc^{s-1, - \ell_0'} = \Sobsc^{s-1,\ell_+ + 1} \subset \cY^{s,\ell_0,\ell_+}$ and the second follows from the previous lemma.

    The claim now follows from Lemma~\ref{lem:finite-dim_kernel} with $P_V^* : \tilde{\cX} \to \tilde{\cY}$ and
    interchanging the roles of $\Rad_{\sources}$ and $\Rad_{\sinks}$.
\end{proof}

If $P_V u \in \CdI(X)$, then we have the following regularity result, which is an immediate consequence of the estimates in Section~\ref{sec:tsc estimates}.
\begin{lemma}\label{lem:global_regularity}
    Let $s, \ell_0, \ell_+, r \in \RR$ with $\ell_0 < -1/2 < \ell_+$, and $r \geq \max\set{1, \ell_+ - \ell_0}$.
    Let $V \in \rho_{\mf} \Psitsc^{1,0}(X)$ be an admissible asymptotically static perturbation of order $r$ and assume that $H_{V_{\pm}}$ have no bound states.
    If $u \in \cX^{s,\ell_0, \ell_+}$ and $P_V u \in \CdI(X)$, then
    for any $\delta > 0$,
    \begin{align*}
        u \in \Sobsc^{\infty, -1/2 - \delta}\,,\text{ and }
        A u \in \Sobsc^{\infty, -1/2 - \delta + r}
    \end{align*}
    provided that $A \in \Psitsc^{0,0}$ satisfies $\WFtsc'(A) \cap \Radtsc_{\sinks} = \varnothing$.
\end{lemma}

The previous lemma states that if $P_V u \in \CdI$, then $u$ is above threshold except at the radial sinks.
If we assume that $P_V u = 0$, then $u$ is above threshold everywhere.

\begin{proposition}\label{prop:smooth_kernel}
    Let $s, \ell_0, \ell_+, r \in \RR$ with $\ell_0 < -1/2 < \ell_+$, and $r \geq \max\set{1, \ell_+ - \ell_0}$.
    Assume that $V$ is an admissible asymptotically static
    perturbation of order $r$,
    and in addition that $V \in \rho_{\mf} \Difftsc^{1}$ and $V = V^*$.
    Then, for all $\delta > 0$,
    \begin{align*}
        \ker_{\cX^{s,\ell_0, \ell_+}}P_V \subset \Sobsc^{\infty,-1/2 - \delta + r}.
    \end{align*}
\end{proposition}
\begin{proof}
    We employ essentially the same argument as the one provided by
    Vasy~\cite{Vasy2000}*{Proposition~17.8} (see also \cite{Vasy2020}*{Proposition~7}).
    For $\alpha \in (-1/2, 0)$ and $\eps > 0$, we introduce the family of
    cut-off functions $\chi_\eps(x)$ given by
    \begin{equation*}
        \chi_\eps(x) = \eps^{-2\alpha-1} \int_{0}^{x/\eps}\varphi(s)^{2}s^{-2\alpha-2}\,ds\,,
    \end{equation*}
    where $\varphi \in \CI(\RR)$ is non-negative,
    $\varphi \equiv 0$ on $\ioc{-\infty, 1}$, and $\varphi \equiv 1$ on $\ico{2, \infty}$.
    For fixed $\eps > 0$, $\chi_{\eps}$ is compactly supported in the
    interior of $M$ and hence an element of $\Psisc^{0,-\infty}(X)$.
    As a family in $\eps \in (0,1)$, however, $\chi_{\eps}$ is not
    uniformly bounded in any symbol space.
    On the other hand, its commutator with $x^{2}\pa_{x}$ is uniformly bounded
    in $\Psisc^{0,2\alpha}$, as
    \begin{equation*}
        [x^{2}\pa_{x}, \chi_{\eps}] = x^{2}\pa_{x}\chi_{\eps}(x) =
        x^{-2\alpha}\varphi\left( x/\eps\right)^{2}.
    \end{equation*}

    For $u\in \ker_{\cX^{s,\ell_0, \ell_+}}P_{V}$, we consider the pairing
    \begin{equation*}
      0 = i\left( \ang{\chi_{\eps} u , P_V u} - \ang{ P_V u, \chi_{\eps}u}\right) = i \ang{ [P_V, \chi_{\eps}] u, u}\,,
    \end{equation*}
    where the second equality follows from $P_V^* = P_V$ and the fact that $\chi_\eps$ is compactly supported.

    Since $V \in \rho_{\mf} \Difftsc^{1}$, we have that $[V, \chi_\eps] = \varphi(x/\eps) f(x, z)$ for some $f \in \Stsc^0$.
    Therefore, we can write
    \begin{align*}
        i[P_V, \chi_\eps] &= 2 x^{-2\alpha} \varphi(x/\eps)^2 (x^2 D_x) + F_\eps\,,
    \end{align*}
    where $F_\eps \in \Psitsc^{0, 2\alpha-1}$ is uniformly bounded.

    We now work in $t \gg 0$, as the argument in $t \ll 0$ is similar
    with a sign change.
    Choose $b \in \CI(\Wperpo)$ such that $b \equiv 1$ for $\tau > m/2$ and $b \equiv 0$ for $\tau < m/4$ and
    set
    \begin{align*}
        B &\coloneqq \Op_L(\tau^{1/2} b(\tau)) G_\psi\,,\\
        E &\coloneqq (x^2 D_x) (\id - G_\psi^2) + \Op_L( \tau(1 - b(\tau)^2)) G_\psi^{2}\,,
    \end{align*}
    so that
    \begin{align*}
        x^2 D_x = B^* B + E + R
    \end{align*}
    with $R \in \Psitsc^{0, - 1}$, hence we obtain that
    \begin{align*}
        i[P_V, \chi_\eps] &= 2 x^{-2\alpha} \varphi(x/\eps)^2 \left( B^* B + E \right) + F_\eps'\,,
    \end{align*}
    where $F_\eps' \in \Psitsc^{0, 2\alpha-1}$ is uniformly bounded.

    We write $E = E_1 + E_2$, where
    \begin{align*}
        E_1 &\coloneqq (x^2 D_x) (\id - G_\psi^2)\,,\\
        E_2 &\coloneqq \Op_L( \tau(1 - b(\tau)^2)) G_\psi^{2}\,.
    \end{align*}
    By Proposition~\ref{prop:elliptic_Gpsi_estimate}, we have that $\norm{E_1 u}_{s, -1/2 - \delta + r} \lesssim \norm{u}_{s-1, -1/2 - \delta}$
    and therefore $E_1 u \in \Sobsc^{\infty,-1/2 - \delta + r}$. The operator $E_2$ has $\tsc$-wavefront set away from the sinks
    of the Hamiltonian flow, $\WFtsc'(E_2) \cap \Radtsc_{\sinks} = \varnothing$, therefore $E_2 u \in \Sobsc^{\infty, -1/2 - \delta + r}$
    by Lemma~\ref{lem:global_regularity}.

    Since commutators with $\varphi(x/\eps)$ decrease the order and are uniformly bounded, we obtain that
    \begin{align*}
        \norm{x^{-\alpha} \varphi(x/\eps) B u}^2 \lesssim \abs{ \ang{ x^{-\alpha} \varphi(x/\eps) E u, x^{-\alpha} \varphi(x/\eps) u} } + \abs{ \ang{F_\eps'' u, u}}
    \end{align*}
    with $F_\eps'' \in \Psitsc^{0, 2\alpha-1}$ being uniformly bounded.
    Taking $\delta = - \alpha \in (0,1/2)$, we observe that the right hand side is finite and hence $x^{-\alpha} \varphi(x/\eps) B u$ is uniformly bounded.
    This implies that $B u$ is bounded in $\Sobsc^{0,\alpha}$ and by Proposition~\ref{prop:elliptic_Gpsi_estimate}, we conclude that
    $u \in \Sobsc^{0, \alpha}$ near $\gamma_{\tsc}(\Radtsc_{\sinks})$, in particular, $u$ is above threshold on $\Radtsc_{\sinks}$.
    This implies that $u$ is in $\Sobsc^{s,-1/2 - \delta + r}$ for $t \gg 0$ by the above threshold radial set estimates.
\end{proof}

With this, we can prove that under the same assumptions
as in the case of the causal propagators,
$P_V : \cX^{s,\ell_0, \ell_+} \to \cY^{s,\ell_0, \ell_+}$ is invertible.

\begin{proof}[Proof of Theorem~\ref{thm:invertible_no_bound}]
  By Lemma~\ref{lem:finite-dim_kernel} and
  Lemma~\ref{lem:finite-dim_cokernel}, we have that $P_V$ is a
  Fredholm operator.  Let $u \in \ker P_V$. By
  Proposition~\ref{prop:smooth_kernel}, $u$ has above threshold
  regularity at the entire radial set. Hence, $u$ is an element in a
  causal space for $P_V$ (for any choice of $s, \ellvar$).  By
  \cite{BDGR}*{Theorem 8.2},
  $P_V : \cX^{s,\ellvar} \to \cY^{s,\ellvar}$ is invertible, hence
  $u \equiv 0$.

    Because Lemma~\ref{lem:finite-dim_cokernel} shows that the
    cokernel is contained in one of the $\cX$ spaces and $P_{V}^{*}$
    satisfies the same estimates, the same argument proves that the cokernel is trivial as well.
\end{proof}

\begin{proposition}\label{prop:unique_inverse}
    The inverse $\PVinv$ is independent of the parameters $s, \ell_0, \ell_+$ and the cut-offs in the following sense:
    let $j \in \set{1,2}$ and $\delta_j > 0$, $Q_{\sources,j} \in \Psitsc_{\sources,\delta_j}$, $Q_{\sources,j}'$, $B_{\sources,j} \in \Psitsc^{0,0}$,
    $\phi_j, \psi_j \in \CcI(\RR)$, and $s_j, \ell_{0,j}, \ell_{+,j} \in \RR$ satisfying the assumptions of Theorem~\ref{thm:invertible_no_bound}.
    Assume that $u_j \in \Sobsc^{s_j,\ell_{0,j}}$ and $Q_{\sources,j} G_{\phi_j} u \in \Sobsc^{s_j, \ell_{+,j}}$
    with $P_V u_j \in \Sobsc^{s_j-1, \ell_{0,j} + 1}$ and $Q_{\sources,j}' G_{\psi_j} P_V u_j \in \Sobsc^{s_j-1, \ell_{+, j} + 1}, B_{\sources,j} P_V u_j \in \Sobsc^{s_j - 2, \ell_{+,j}}$.
    If $P_V u_1 = P_V u_2$, then $u_1 = u_2$.

    In particular, $\PVinv$ defines a continuous linear operator
    \begin{align*}
        \PVinv : \CdmI(X) \to \CmI(X)\,.
    \end{align*}
\end{proposition}
\begin{proof}
    Set $u \coloneqq u_1 - u_2$. We have that $P_V u = 0$ and we can find parameters $s, \ell_0, \ell_+ \in \RR$ and cut-offs such that
    $u$ is in a $\cX$-space. Since $P_V : \cX \to \cY$ is invertible, it follows that $u \equiv 0$.
\end{proof}

\section{Regularity of solutions with compactly supported forcing}\label{sec:regularity}

We now relate the propagators constructed here to the distinguished parametrices
of Duistermaat--Hörmander~\cite{DH1972}.  We follow the exposition
of G\'{e}rard--Wrochna~\cites{GW2019,GW2020}.\footnote{Note that the sign conventions are different, hence the flow direction is reversed.}

A bounded linear operator $A : \CdI(X) \to \CmI(X)$ can be identified with its Schwartz kernel $K \in \CmI(X \times X)$. We set
\begin{align*}
    \widetilde{\WF}(A) \coloneqq \set*{ (p, \xi, p', \xi') \colon (p, p', \xi, -\xi') \in \WF_{\cl}(K)}\,.\footnotemark
\end{align*}
\footnotetext{This set is usually denoted by $\WF'$, but this causes confusion with the operator wavefront set.}%
The microlocal Hadamard condition as introduced by Radzikowski~\cite{Radzikowski1996} is the following:
\begin{definition}
    A bounded linear operator $E : \CcI \to \CmI$ is called a Feynman parametrix if $\id - E P_V$ and $\id - P_V E$ have smooth Schwartz kernels and
    \begin{align*}
        \widetilde{\WF}(E) = \diag_{T^* X \setminus 0} \cup \bigcup_{s \geq 0} \tilde{\Phi}_s( \diag_{\Char(P_0)} ) \,.
    \end{align*}
    Here, $\tilde{\Phi}_s$ denotes the bicharacteristic flow acting on the left component of $T^* X \times T^* X$.
\end{definition}

\begin{proposition}
    Let $E$ be a Feynman parametrix. Then $E$ uniquely extends to a map $\CcmI \to \CmI$ and for $f \in \CcmI$ we have that
    \begin{align*}
        \WF_{\cl}(E f) \subset \WF_{\cl}(f) \cup \bigcup_{s \geq 0} \Phi_s\left( \WF_{\cl}(f) \cap \Char(P_0) \right)\,.
    \end{align*}
\end{proposition}
\begin{proof}
    This follows almost directly from Hörmander~\cite{Hormander1}*{Theorem~8.2.13}. Denote by $K_E$ the Schwartz kernel of $E$.
    By the condition on $\widetilde{\WF}(E)$, we have that
    \begin{align*}
        \WF'(K_E)_X \coloneqq \set{(p', \xi') \colon (p, p', 0, -\xi') \in \WF_{\cl}(K_E) \text{ for some } p \in X} = \varnothing,
    \end{align*}
    hence $E f$ is uniquely defined for any $f \in \CcmI(X)$.
    Moreover, we have that $\WF(K_E)_X \coloneqq \set{(p, \xi) \colon (p, p', \xi, 0) \in \WF_{\cl}(K_E) \text{ for some } p' \in X} = \varnothing$, so
    \begin{align*}
        \WF_{\cl}(E f) &\subset \widetilde{\WF}(E) \circ \WF_{\cl}(f) \\
        &= \WF_{\cl}(f) \cup \set{ (\exp(s \Hamsc_p) q, q) \colon q \in \WF_{\cl}(f) \cap \Char(P_0)}\,.
    \end{align*}
\end{proof}

This version of the microlocal Hadamard condition is more in line with the support condition for the causal propagators. 
The inverse $\PVinv$ as constructed in Proposition~\ref{prop:unique_inverse} has the same property:
\begin{theorem}\label{thm:regularity}
    Let $f \in \CdmI(X)$, then
    \begin{align*}
        \WF_{\cl}(\PVinv f) \subset \WF_{\cl}(f) \cup \bigcup_{s \geq 0} \Phi_s \left( \WF_{\cl}(f) \cap \Char(P_0) \right)\,.
    \end{align*}
\end{theorem}
This theorem follows directly from the propagation estimates in the scattering calculus.

\begin{proof}[Proof of Theorem~\ref{thm:main}]
    Theorem~\ref{thm:main} is a direct consequence of Theorem~\ref{thm:invertible_no_bound}, Proposition~\ref{prop:unique_inverse}, and Theorem~\ref{thm:regularity}.
\end{proof}

\section{The case of bound states}\label{sec:bound states}

We now discuss the case that one or both of the limiting spatial
Hamiltonians $H_{V_{\pm}}$ admits bound states.  This follows the
treatment of bound states in our work~\cite{BDGR}*{Sect.\ 8.2} on the causal propagators.
We now assume that there are only finitely many bound states of the limiting Hamiltonians, $H_{V_\pm}$ and
they are contained in the interval $\ioo{-\infty, m^2}$.
Moreover,
we assume throughout that
\begin{equation}
  \label{eq:no thresh eigenvalue}
  0 \not \in \sigma(H_{V_{\pm}}),
\end{equation}
i.e., $0$ is not an eigenvalue of $H_{V_{\pm}}$, as the corresponding
solutions with linear growth in time complicate the analysis
substantially.

Setting $\lambda(\tau) = \sqrt{m^{2} - \tau^{2}}$, denote the
eigenspace of $\Delta_{z} + V_{\pm}$ at frequency $\lambda$ by
\begin{equation*}
  E_{\pm}(\lambda(\tau_{0})) = \set*{ w \in L^{2}(\RR^{n}) \colon
    \ffsymbzpm(P_V)(\tau_{0}) w = 0}.
\end{equation*}
Since the operators $\ffsymbzpm(P_V)(\tau_{0})$ are scattering
elliptic for $|\tau_{0}| < m$, such bound states $w$ are Schwartz:
\begin{equation*}
  E_{\pm}(\lambda(\tau_{0})) \subset \schwartz(\mathbb{R}^{n}).
\end{equation*}
We denote the set of values $\tau$ at which there are non-trivial
bound states by
\begin{equation}
  \label{eq:boundstates}
  \cB_{V,\pm} = \{ \tau \in (-m, m) \colon E_{\pm}(\lambda(\tau)) \neq
  \{ 0 \} \}.
\end{equation}
Thus, under our assumptions, $\cB_{V,\pm}$ is finite and $0 \not \in
\cB_{V,\pm}$.  An element $\tau_{0} \in \cB_{V,+}$ and a choice of $w \in
E_{+}(\lambda(\tau))$ gives solutions
\begin{equation}\label{eq:asymptotic soln bound}
u_{w, \tau_0}(t, z) \coloneqq e^{i \tau_{0} t} w(z), \quad   P_{V_{+}}(u_{w, \tau_0}) = 0,
\end{equation}
and the same for $V_{-}$; these solutions lie exactly at threshold,
meaning $e^{\pm i \tau_{0} t} w(z) \in \Sobsc^{- \infty, -1/2 -
  \epsilon}$ for all $\epsilon > 0$ and no better.  They have
wavefront set at $\tau_{0} \in W^{\perp}_{+}$.

\begin{figure}
    \centering
    \begin{tikzpicture}[scale=2]
        \draw (0,1) -- (0,-1);
        \draw [thick] (0,1) -- (0,0.75);
        \draw [thick] (0,-0.75) -- (0,-1);
        \filldraw (0,0.75) circle (0.5pt) node [right] {$\tau = m$};
        \filldraw (0,-0.75) circle (0.5pt) node [right] {$\tau = -m$};
        \draw (-0.03,0) -- (0.03,0);
        \draw plot [only marks, mark=x, mark size=1pt] coordinates
            {(0,0.5) (0,0.2) (0,0.1) (0,-0.1) (0,-0.2) (0,-0.5)};
    \end{tikzpicture}
    \caption{The set $\cB_{V,+} \subset W_+^\perp$}
\end{figure}
From the discussion in the introduction, it is clear which asymptotic
solutions $e^{i \tau_{0} t} w(z)$ should be allowed in the domain and
which should be excluded (see \eqref{eq:free asymptotics}) to obtain
a Feynman type problem.
Near $\NP$ we allow asymptotics as in
\eqref{eq:asymptotic soln bound} with $\tau_{0} > 0$ and exclude those
with $\tau_{0} < 0$, while near $\NP$ we allow asymptotics as in
\eqref{eq:asymptotic soln bound} with $\tau_{0} < 0$ and exclude those
with $\tau_{0} > 0$.
We set
\begin{align*}
    \cB_{V, \sources} &\coloneqq \set{ (\fut, \tau) \colon \tau \in \cB_{V,+}, \tau < 0} \cup \set{ (\pas, \tau) \colon \tau \in \cB_{V,-}, \tau > 0}\,,\\
    \cB_{V, \sinks} &\coloneqq \set{ (\fut, \tau) \colon \tau \in \cB_{V,+}, \tau > 0} \cup \set{ (\pas, \tau) \colon \tau \in \cB_{V,-}, \tau < 0}\,.
\end{align*}

To quantify this inclusion/exclusion statement, we introduce the Fourier localizing bound
state projectors from \cite{BDGR}*{Sect.\ 8.2}.  Specifically,
following \cite{BDGR}*{Eq.\ 8.21}, we let $\Kf_{\tau_{0}} =
K_{\tau_{0}}(t, z, t', z')$ denote the integral kernel of the operator
which: (1) cuts off to large positive times, (2) projects onto the
space of bound states $E_{+}(\lambda(\tau_{0}))$, and (3) localizes in
$\tau$ Fourier space near $\tau_{0}$.  Specifically
\begin{equation}
  \label{eq:K future tau 0}
  \Kf_{\tau_{0}}(t, z, t', z') = \frac{1}{2 \pi} \int_{-\infty}^{\infty}
e^{i(t - t') \tau} \chi_{\ge t_{0}}(t) \chi_{\tau_{0}} (\tau)   \cdot
\Pi_{\tau_{0}}(z, z') \chi_{\ge t_{0}}(t') d\tau.
\end{equation}
(Here $\chi_{\ge t_{0}}(t)$ is a cutoff to $t \ge t_{0}$,
$\Pi_{\tau_{0}}$ is the projection onto  $E_{+}(\lambda(\tau_{0}))$,
and $\chi_{\tau_{0}}$ is a cutoff to a small neighborhood of
$\tau_{0}$.)
Near $\SP$ we use $\Kp_{\tau_{0}}$, defined similarly but with
$\chi_{\ge t_{0}}$ replaced by $\chi_{\le t_{0}}$ which cut off to
negative times.

Importantly, we have:
\begin{lemma}[cf. \cite{BDGR}*{Lemma 8.9}]
    For $\chi_{\tau_{0}}$ with sufficiently small support near $\tau_{0}$,
    \begin{equation*}
        \Kfp_{\tau_{0}} \in \Psitsc^{-\infty, 0}.
    \end{equation*}
\end{lemma}

To ensure that localizers $Q \in \Psitsc_{\sources,\delta}$ as defined in Definition~\ref{def:localizer-to-src}
do not modify the eigenspaces $E_{\pm}(\lambda(\tau))$, we will assume throughout this section
that $\WF_{\ff}'(QG_{\phi})$ is disjoint from the set of
eigenvalues of the limiting Hamiltonian, i.e.,
\begin{align*}
    \WF_{\ff}'(QG_\varphi) \cap \cB_{V,\pm} = \varnothing\,.
\end{align*}
The construction in Lemma~\ref{lem:there-are-localizers} can be easily modified by shrinking the support of $f$ to satisfy this additional condition.

Our spaces now include both the condition that the distributions
be above threshold at the sources, but also that the projections onto
the appropriate bound states be above threshold.  Specifically, if
again we have $s, \ell_0, \ell_+ \in \RR$ with
$\ell_0 < -1/2 < \ell_+ < 1/2$ and we let $Q_{\sources} G_\phi$ denotes the
source microlocalizer from Section \ref{sec:tsc localizers} with the modification just mentioned, and
$u \in H^{s, \ell_0}$ then Feynman-type distributions should satisfy
\begin{equation} \label{eq:BS Feynman condition}
    Q_{\sources} G_\phi u \in \Sobsc^{s, \ell_+},\ K^\kappa_{\tau} u \in \Sobsc^{s,\ell_+} \text{ for all } (\kappa, \tau) \in \cB_{V, \sources}\,.
\end{equation}
These conditions can be encapsulated in a $H^{s, \ell_0, \ell_+}_A$
space for appropriate $A$.  Indeed, let
\begin{equation} \label{eq:conditions together}
    \Qsrcpsi{\phi} = Q_{\sources} G_\phi + \cK_{\sources}
\end{equation}
where
\begin{equation}
    \cK_{\sources} \coloneqq \sum_{(\kappa, \tau) \in \cB_{\sources}} K^\kappa_{\tau}\,.
\end{equation}
If we (easily) arrange that all of the terms on the RHS have disjoint
operator wavefront set, then $\Qsrcpsi{\psi} u \in \Sobsc^{s, \ell_+}$ if and only
if the condition in \eqref{eq:BS Feynman condition} holds.

We now define our spaces as in the previous sections.  Since we use the same
method for estimating the $Q_{\sources}' G_\psi u$ terms, we also
use the notation $\Qsrcpsi{\phi}'$ for the operator as in
\eqref{eq:conditions together} whose RHS has $Q_{\sources}' G_{\psi}$
replacing $Q_{\sources} G_\phi$.  Thus we assume again
that $\WF' (Q_{\sources}) \subset \Ell(Q_{\sources}')$,
$\phi, \psi \in \CcI(\RR)$ are bump functions supported at $0$
with $\phi \psi = \phi$, and $B_{\sources} \in \Psitsc^{0,0}$ with $\Rad_{\ff,\sources} \subset \Ell_{\ff}(B_{\sources})$.
Define
\begin{equation}\label{eq:XY spaces BS}
\begin{aligned}
    \cY^{s,\ell_0,\ell_+} &\coloneqq \set*{ v \in \Sobsc^{s-1, \ell_0+ 1} \colon \Qsrcpsi{\psi}' v \in \Sobsc^{s-1, \ell_+ + 1}, B_{\sources} v \in \Sobsc^{s-2, \ell_+}}\,,\\
    \cX^{s,\ell_0,\ell_+} &\coloneqq \set*{ u \in H_{\Qsrcpsi{\phi}}^{s,\ell_0,\ell_+}  \colon P_V u \in \cY^{s,\ell_0, \ell_+} }\,.
\end{aligned}
\end{equation}

In our treatment of bound states, for technical reasons, we assume that
$V - V_{\pm}$ decays one order faster than the assumption used
above.  We have the following main theorems in the presence of bound states.
\begin{theorem}\label{thm:fredholm_strong_decay}
    Let $s, \ell_0, \ell_+, r \in \RR$ with $\ell_0 < -1/2 < \ell_+ < 1/2$ and $r \geq \max\set{2, \ell_+ - \ell_0}$.
    Let $V \in \rho_{\mf} \Psitsc^{1,0}$ be asymptotically static of order $r$ and
    assume the limiting Hamiltonians have purely absolutely continuous spectrum in $\ico{m^2, \infty}$, and finitely many bound states in $\ioo{-\infty, m^2}$
    and \eqref{eq:no thresh eigenvalue} holds.
    Moreover, assume that
    \begin{align*}
        V - V^* \in \Psitsc^{0,-2}(\RR^{n+1})\,.
    \end{align*}
    Then
    \begin{align*}
        P_V : \cX^{s,\ell_0, \ell_+} \to \cY^{s,\ell_0, \ell_+}
    \end{align*}
    is Fredholm.
\end{theorem}

\begin{theorem}\label{thm:invertible strong decay}
    Let $s, \ell_0, \ell_+, r \in \RR$ with $\ell_0 < -1/2 < \ell_+ < 1/2$ and $r \geq \max\set{2, \ell_+ - \ell_0}$.
    Let $V \in \rho_{\mf} \Difftsc^{1,0}$ be asymptotically static of order $r$ and
    \begin{itemize}
        \item $V$ is self-adjoint, $V = V^*$,
        \item limiting Hamiltonians have purely absolutely continuous spectrum in $\ico{m^2, \infty}$, and finitely many bound states in $\ioo{0, m^2}$,
            in particular $H_{V_\pm} > 0$.
    \end{itemize}
    then $P_V : \cX^{s,\ell_0, \ell_+} \to \cY^{s,\ell_0, \ell_+}$ is
    invertible.
    We denote the inverse by $\PVinv$ and for $f \in \CdmI$ the function $\PVinv f \in \CmI$ is independent of the choice of parameters $s,\ell_0, \ell_+$ and microlocal cutoff $Q G_\phi$.
    For any $f \in \CdmI$, we have
    \begin{align*}
        \WFtsc(\PVinv f) \subset \WFtsc(f) \cup \bigcup_{s \geq 0} \Phitsc_s \left( \WFtsc(f) \cap \Chartsc(P_0) \right) \cup \Radtsc_{\sinks} \cup \cB_{V, \sinks}\,.
    \end{align*}

    The same conclusion holds if $V = V^{*} \in
    \rho_{\mf}\Difftsc^{1}$ is globally static, i.e., $V = V(x)$.  In
    particular the positivity condition $H_{V} > 0$ is not needed in
    the static case, though we still assume $0 \not \in \cB_{V}$.
\end{theorem}
To prove the Fredholm result we again use a global Fredholm
estimate, and the main result we use from our previous work
regarding the $\Kfp_{\tau}$ is the following lemma.
\begin{lemma}\label{thm:elliptic on kernel}
  Let $\ell \in \RR$, $\tau_0 \in \cB_{V,\pm}$.  Below we state the estimate near $\NP$, but the
  same is true near $\SP$ with appropriate $+$'s replaced by $-$'s.

  Below threshold: Assume $\ell < -1/2$.
  Then for any $s, M, N \in \RR$ there is $C > 0$ such that, provided $\Kf_{\tau_0} P_V u \in \Sobsc^{s, \ell + 1}$, then we have $\Kf_{\tau_0} u \in \Sobsc^{s, \ell}$
  \begin{equation}
      \norm{\Kf_{\tau_0} u}_{s, \ell} \le
      C( \norm{\Kf_{\tau_0} P_Vu}_{s,\ell + 1} +  \normres{u})\,.\label{eq:ODE estimate}
  \end{equation}

  Above threshold: Assume $\ell > -1/2$.  For any
  $s, \ell', M, N \in \RR$ with $\ell > \ell' > -1/2$, there is $C$
  such that, if $\Kf_{\tau_0} u \in \Sobsc^{s, \ell'}$ and
  $P_V \Kf_{\tau_0} u \in \Sobsc^{s, \ell + 1}$, then $\Kf_{\tau_0}u \in \Sobsc^{s, \ell}$ and
  \begin{equation}
      \norm{\Kf_{\tau_0} u}_{s, \ell} \le C(
      \norm{\Kf_{\tau_0} P_V u}_{s,\ell + 1} +
      \norm{\Kf_{\tau_0} u}_{-N,\ell'} + \normres{u})\,.\label{eq:ODE estimate2}
  \end{equation}
\end{lemma}
  
\begin{proof}[Proof of Fredholm property]
  To get the closed range and finite dimensional kernel statement, we
  can use a global Fredholm estimate in Lemma
  \ref{lem:finite-dim_kernel} but for the $\cX, \cY$ spaces in
  \eqref{eq:XY spaces BS}, where now
  \begin{equation*}
        \norm{v}_{\cY^{s,\ell_0,\ell_+}}
        \sim
    \norm{v}_{s-1,\ell_0 + 1} +
    \norm{Q_{\sources}' G_{\psi} v}_{s - 1,\ell_+ + 1} +
    \norm{B_{\sources} v}_{s - 2,\ell_+} +
    \norm{\cK_{\sources} v}_{s - 1,\ell_+ + 1} \,.
  \end{equation*}
  As in the case of the causal propagators \cite{BDGR}*{p.~91}, we modify the partition of unity by introducing additional microlocalizers $B_{5,\pm}$ and $B_{6,\pm}$
  such that $B_{5,\pm}$ is elliptic on $(-m + \eps', \delta') \subset W^\perp_{\pm}$ and $B_{6,\pm}$ is elliptic on $(-\delta', m - \eps')$.
  For the microlocalizers $B_{5,\pm}$ and $B_{6,\pm}$, we use the estimates from Lemma~\ref{thm:elliptic on kernel}.
  We arrive at the estimate
  \begin{align*}
      \norm{u}_{\cX^{s,\ell_0,\ell_+}} \lesssim \norm{P_V u}_{\cY^{s,\ell_0,\ell_+}} + \normres{u}
  \end{align*}
  from which we deduce finite dimensionality of the kernel and that the range is closed.

  By the same argument as in Lemma~\ref{lem:finite-dim_cokernel}, we obtain that cokernel is finite dimensional as well, which implies that $P_V$ is Fredholm as claimed.
\end{proof}

\begin{proposition}\label{prop:smooth_kernel-bound_states}
    Let $V$ satisfy the assumptions in Theorem \ref{thm:invertible
      strong decay}. Then, for some $\delta > 0$,
    \begin{align*}
        \ker_{\cX^{s,\ell_0, \ell_+}}P_V \subset \Sobsc^{s,-1/2 + \delta}.
    \end{align*}
\end{proposition}
\begin{proof}
The proof is identical to that of
Proposition~\ref{prop:smooth_kernel} with the modification that, for the
operator $B = \Op_{L}(\tau^{1/2}b(\tau))G_{\psi} $, we must have $b
\equiv 1$ on $\cB_{V, \sinks} \cup \Rad_{\ff,{}\sinks}$ and $\supp b
\cap (-\infty, 0] = \varnothing$.  This is easily arranged, for example by choosing $c > 0$
with $c < \min \set{ \abs{\tau_{0}} : \tau_{0} \in \cB_{V, \sinks}}$.  Then
for example at $\NP$ we can take $b(\tau) \equiv 1$ for $c \le \tau
\le m + c$ and $b(\tau) = 0$ for $\tau < c / 2$ and $\tau > m + 2c$.
\end{proof}

\begin{proof}[Proof of invertibility]
    We now assume that the spatial Hamiltonians satisfy
    \begin{equation}\label{eq:positive}  
        H_{V_{\pm}} > 0.
    \end{equation}
    The proof that the kernel and cokernel are trivial is the same as
    for the case with no bound states.
    By Proposition~\ref{prop:smooth_kernel-bound_states}, we know that every $u \in \ker P_V$
    is a Schwartz function and therefore satisfies the above threshold
    decay rate on the entire radial set.
    Hence $u$ is an element in a causal Sobolev space $\cX^{s,\ellvar}$ for any choice of
    $s, \ellvar$ and by the theorem for causal propagators~\cite{BDGR}*{Theorem~8.3}
    the operator $P_V : \cX^{s,\ellvar} \to \cY^{s,\ellvar}$ is invertible and
    therefore $u \equiv 0$.

    The uniqueness and the wavefront property follow from the same arguments as in the case of no bound states.

    In the static case one sees directly that there are no global
    elements in the kernel or cokernel as their projection onto the
    bound states would have the excluded asymptotics at one of the two
    time-like infinities.
\end{proof}

\end{document}